\documentclass[]{interact}

\usepackage{epstopdf}
\usepackage[caption=false]{subfig}

\usepackage[numbers,sort&compress]{natbib}
\bibpunct[, ]{[}{]}{,}{n}{,}{,}
\renewcommand\bibfont{\fontsize{10}{12}\selectfont}
\usepackage{amsmath}
\usepackage{amssymb}
\usepackage[mathscr]{euscript}
\usepackage{enumerate}
\usepackage{xspace}
\usepackage{color}
\usepackage{enumerate}
\usepackage{enumitem}
\usepackage{amsthm}

\theoremstyle{plain}
\newtheorem{theorem}{Theorem}[section]
\newtheorem{lemma}[theorem]{Lemma}
\newtheorem{corollary}[theorem]{Corollary}
\newtheorem{proposition}[theorem]{Proposition}

\theoremstyle{definition}
\newtheorem{definition}[theorem]{Definition}
\newtheorem{example}[theorem]{Example}

\theoremstyle{remark}
\newtheorem{remark}{Remark}[section]

\newtheorem*{comment}{Comment}
\newtheorem{condition}{Condition}
  
 \newcommand{\of}{[\hspace{-0.06cm}[}
 \newcommand{\gs}{]\hspace{-0.06cm}]}
 \newcommand{\dd}{d}
 \newcommand{\cadlag}{c\`adl\`ag }
 \newcommand{\1}{\mathbf{1}}
 \newcommand{\lle}{\langle\hspace{-0.085cm}\langle}
 \newcommand{\rre}{\rangle\hspace{-0.085cm}\rangle}

\begin{document}


\title{On the Existence of Semimartingales with Continuous Characteristics}

\author{
\name{David Criens\textsuperscript{a}\thanks{CONTACT D. Criens. Email: \texttt{david.criens@tum.de}}}
	\affil{\textsuperscript{a} Center for Mathematics, Technical University of Munich, Munich, Germany}
}

\maketitle

\begin{abstract}
We prove the existence of quasi-left continuous semimartingales with continuous local semimartingale characteristics which satisfy a Lyapunov-type or a linear growth condition, where latter takes the whole history of the paths into consideration. The proof is based on an approximation and a tightness argument and the martingale problem method. 
\end{abstract}



\section{Introduction}
Existence theorems for solutions to stochastic equations are of fundamental interest in many areas of probability theory. In the context of \emph{weak solutions} to stochastic differential  equations (SDEs) important contributions were made by Skorokhod and by Stroock and Varadhan. Skorokhod (see \cite{skorokhod2014studies}) showed that SDEs with continuous coefficients of linear growth have weak solutions. Stroock and Varadhan (see \cite{SV}) introduced the concept of the \emph{martingale problem}, which is nowadays one of the most important tools for studying existence, uniqueness and limit theorems for stochastic processes.
In many of the classical monographs on stochastic analysis (e.g., \cite{KaraShre,RY}) Skorokhod's existence theorem is proven by the martingale problem argument of Stroock and Varadhan.
The main idea is to construct an approximation sequence of probability measures on a path space, to show its tightness and finally to use the martingale problem method to verify that any of its accumulation points is the law of a weak solution.

In case of SDEs with Wiener noise and coefficients of linear growth, tightness can be verified via Kolmogorov's tightness criterion. 
Gatarek and Goldys \cite{doi:10.1080/17442509408833868} proposed a more direct argument for tightness based on the compactness of a  fractional operator and the factorization method of Da Prato, Kwapien and Zabczyk \cite{doi:10.1080/17442508708833480}. 
This method was used by Hofmanov\'a and Seidler \cite{doi:10.1080/07362994.2013.799025} to replace the linear growth assumption in Skorokhod's theorem by a Lyapunov-type condition. 

Skorokhod's original theorem is not restricted to path continuous settings. For general semimartingales Jacod and M\'emin \cite{PSMIR_1979___1_A4_0} proved conditions for tightness in terms of the so-called \emph{semimartingale characteristics}. These criteria were used by Jacod and M\'emin \cite{doi:10.1080/17442508108833169} to prove continuity and uniform boundedness conditions for the existence of weak solutions to SDEs driven by general semimartingales. 

Refinements of the tightness criteria from \cite{PSMIR_1979___1_A4_0} are proved in the monograph \cite{JS} of Jacod and Shiryaev. The conditions are used to prove a Skorokhod-type existence result for semimartingales. More precisely, Jacod and Shiryaev consider a candidate for semimartingale characteristics on the Skorokhod space and formulate continuity and uniform boundedness conditions which imply the existence of a probability measure for which the coordinate process is a semimartingle with the candidate as semimartingale characteristics. 

In this article we generalize the existence result of Jacod and Shiryaev for the quasi-left continuous case by replacing the uniform boundedness assumption by local boundedness assumptions together with a Lyapunov-type or a linear growth condition. The linear growth condition takes the whole history of the paths into consideration. We prove the result as follows: First, we construct an approximation sequence with the help of the existence result of Jacod and Shiryaev. Second, we show tightness by a localization of a criterion from \cite{JS} together with a Lyapunov-type or a Gronwall-type argument. In this step we also adapt arguments used by Liptser and Shiryaev \cite{liptser1989theory}. Finally, we use arguments based on the martingale problem for semimartingales to verify that any accumulation point of our approximation sequence is the law of a semimartingale with the correct semimartingale characteristics.

Let us shortly comment on continuative problems. The weak convergence argument heavily relies on the continuous mapping theorem, which is applicable when the coefficients have a continuity property. It is only natural to ask what can be said for discontinuous coefficients. We do not touch this topic in the present article and refer the curious reader to the recent articles \cite{IMKELLER2016703,kuhn18} where interesting progress in this direction is made. 

The article is structured as follows. In Section \ref{sec: 2.1} we explain the mathematical setting of the article. In Section \ref{sec: 2.2} we state our main results. In particular, we discuss its assumptions. Finally, we comment on the method based on the extension of local solutions and on a possible expansion of our result via Girsanov-type arguments. In Section \ref{sec: 2.3} we apply our results in a jump-diffusion setting. The proofs of our main results are given in Section \ref{sec:p}. 

The topic of this article is of course very classical and the basic definitions can be found in many textbooks. Our main reference is the monograph of Jacod and Shiryaev \cite{JS}. As far as possible we will refer to results in this monograph. Furthermore, all non-explained terminology can also be found there.

\section{Formulation of the Main Results}\label{sec: main}
\subsection{The Mathematical Setting}\label{sec: 2.1}
Let \(\Omega\) be the Skorokhod space of \cadlag functions \(\mathbb{R}_+ \to \mathbb{R}^d\) equipped with the Skorokhod topology (see \cite{JS} for details). We denote the coordinate process on \(\Omega\) by \(X\), i.e. \(X_t(\omega) = \omega(t)\) for \(t \in \mathbb{R}_+\) and \(\omega \in \Omega\). Let \(\mathcal{F} \triangleq \sigma(X_t, t \in \mathbb{R}_+)\) and 
\(
\mathcal{F}_t \triangleq \bigcap_{s > t} \mathcal{F}^o_s, 
\)
where \(\mathcal{F}^o_s \triangleq \sigma(X_t, t \in [0, s])\). Except stated otherwise, when we use terms such as \emph{adapted, predictable, etc.} we refer to the right-continuous filtration \((\mathcal{F}_t)_{t \geq 0}\).

Throughout the article we fix a continuous truncation function \(h \colon \mathbb{R}^d \to \mathbb{R}^d\), i.e. a bounded continuous function which equals the identity around the origin.

A \cadlag \(\mathbb{R}^d\)-valued adapted process \(Y\) is called a semimartingale if it admits a decomposition 
\(
Y = Y_0 + M + V,
\)
where \(M\) is a \cadlag local martingale starting at the origin and \(V\) is a \cadlag adapted process of finite variation starting at the origin. Here, we adapt the terminology from \cite{JS} and call a process \(V\) of finite variation if for all \(\omega \in \Omega\) the map \(t \mapsto V_t(\omega)\) is locally of finite variation.
To a semimartingale \(Y\) we associate a quadruple \((b, c, K; A)\) consisting of an \(\mathbb{R}^d\)-valued predictable process \(b\), a predictable process \(c\) taking values in the set \(\mathbb{S}^d\) of symmetric non-negative definite \(d\times  d\) matrices, a predictable kernel \(K\) from \(\Omega \times \mathbb{R}_+\) into \(\mathbb{R}^ d\) and a predictable increasing \cadlag process \(A\), see \cite[Definition II.2.6, Proposition II.2.9, II.2.12 -- II.2.14]{JS} for precise definitions and properties.
When \((B, C, \nu)\) are the semimartingale characteristics of \(Y\) (see \cite[Definition II.2.6]{JS}), then 
\[
\frac{\dd B_t}{\dd A_t} = b_t, \qquad \frac{\dd C_t}{\dd A_t} = c_t, \qquad \frac{\nu(\dd t, \dd x)}{\dd A_t} = K_t(\dd x), 
\]
i.e. in other words \((b, c, K)\) are the densities of \((B, C, \nu)\) w.r.t. the reference measure \(\dd A_t\). Thus, we call the quadruple \((b, c,  K; A)\) the \emph{local characteristics of \(Y\)}.
Providing an intuition, \(b\) represents the drift and depends on the  truncation function \(h\), \(c\) encodes the continuous local martingale component and \(K\) reflects the jump structure. 
In addition, for \(i, j = 1, \dots, d\) we define by
\[
\widetilde{c}^{ij} \triangleq c^{ij}  + \int h^i(x) h^j(x) K(\dd x) - \Delta A \int h^i(x) K(\dd x) \int h^j(x) K(\dd x) 
\]
a \emph{modified second characteristic}, see \cite[Proposition II.2.17]{JS}.

Let us shortly comment on the role played by the initial law. 
For SDEs with Wiener noise it was proven by Kallenberg \cite{10.2307/2244838} that weak solutions exist for all initial laws if, and only if, weak solutions exist for all degenerated initial laws. Although the result is fairly old, it seems not to be commonly known.
We now state a version for a general semimartingale setting. The proof is similar as in the diffusion case and can be found in Appendix \ref{sec:p2}.
\begin{proposition}\label{prop: ex}
	Assume that for all \(z \in \mathbb{R}^d\) there exists a probability measure \(P_z\) on \((\Omega, \mathcal{F})\) such that the coordinate process is a \(P_z\)-semimartingale with local characteristics \((b, c, K; A)\) and initial law \(\delta_z\). Then, for any Borel probability measure \(\eta\) on \(\mathbb{R}^d\) there exists a probability measure \(P_\eta\) on \((\Omega, \mathcal{F})\) such that the coordinate process is a \(P_\eta\)-semimartingale with local characteristics \((b, c, K; A)\) and initial law \(\eta\). 
\end{proposition}

From now on we fix a deterministic continuous increasing function \(A \colon \mathbb{R}_+ \to \mathbb{R}_+\) with \(A_0 = 0\) and a Borel probability measure \(\eta\) on \(\mathbb{R}^d\).
Next, we define a so-called \emph{candidate triplet} \((b, c, K)\) on \((\Omega, \mathcal{F})\). Let us shortly clarify some notations: For \(x, y \in \mathbb{R}^d\) we write \(\|x\|\) for the Euclidean norm, \(\langle x, y\rangle\) for the Euclidean scalar product, and for \(M \in \mathbb{S}^d\) we write \(\|M\| \triangleq \textup{trace } M\). 
\begin{definition}
	We call \((b, c, K)\) a candidate triplet, if it consists of the following:
	\begin{enumerate}
		\item[\textup{(i)}] A predictable \(\mathbb{R}^d\)-valued process \(b\) such that \(\int_0^t \|b_s(\omega)\| \dd A_s < \infty\) for all \((t,\omega) \in \mathbb{R}_+ \times \Omega\).
		\item[\textup{(ii)}]  A predictable \(\mathbb{S}^d\)-valued process \(c\) such that \(\int_0^t \|c_s(\omega)\| \dd A_s < \infty\) for all \((t,\omega) \in \mathbb{R}_+ \times \Omega\).
		\item[\textup{(iii)}]  A predictable kernel \((\omega, s) \mapsto K_s(\omega;\dd x)\) from \(\Omega \times \mathbb{R}_+\) into \(\mathbb{R}^d\) such that for all \((t,\omega) \in \mathbb{R}_+ \times \Omega\) we have \(K_{t}(\omega; \{0\}) = 0\) and \(\int_0^t \int (1  \wedge \|x\|^2) K_{s}(\omega; \dd x) \dd A_s < \infty\).
	\end{enumerate}
\end{definition}
In the following we fix also a candidate triplet \((b, c, K)\). 
The goal is to find a probability measure \(P\) on \((\Omega, \mathcal{F})\) such that the coordinate process \(X\) is a \(P\)-semimartingale with local characteristics \((b, c, K; A)\) and initial law \(\eta\).
\subsection{Existence Conditions for Semimartingales}\label{sec: 2.2}

Let \(C_2(\mathbb{R}^d)\) be the set of all continuous bounded function \(\mathbb{R}^d \to \mathbb{R}\) which vanish around the origin. 
Moreover, let \(C_1(\mathbb{R}^d)\) be a subclass of the non-negative functions in \(C_2(\mathbb{R}^d)\) which contains all functions \(g (x) = (a \|x\| - 1)^+ \wedge 1\) for \(a \in \mathbb{Q}\) and is convergence determining for the weak convergence induced by \(C_2(\mathbb{R}^d)\) (see \cite[p. 395]{JS} for more details).

For a twice continuously differentiable function \(f \colon \mathbb{R}^d \to  \mathbb{R}\) and \(a > 0\) we set 
\begin{align*}
\widetilde{c}^{ij, a} &\triangleq c^{ij} + \int_{\|x\| \leq a} x^i x^j  K(\dd x),\quad b^a \triangleq b - \int \big(h(x) - x \1 \{\|x\| \leq a\}\big)K(\dd x)
\end{align*}
and for all \((t, \omega) \in \mathbb{R}_+ \times \Omega\) and \(x \in \mathbb{R}^d\) we set
\begin{align*}
(\mathcal{K}^d_a f)(\omega; t, x) &\triangleq f(\omega(t-) + x) - f(\omega(t-)) - \sum_{k =  1}^d  \partial_{k} f(\omega(t-)) x^k, 
\\
(\mathcal{K}^l_a f)(\omega; t) &\triangleq \sum_{k = 1}^d \partial_{k} f(\omega(t-)) b^{k, a}_t (\omega) + \frac{1}{2} \sum_{k, j = 1}^d \partial^2_{kj} f (\omega(t-)) c^{kj}_t  (\omega),
\end{align*}
and
\begin{align*}
(\mathcal{L}_a f) (\omega; t) \triangleq 
(\mathcal{K}^l_a f)(\omega; t) + \int_{\|x\| \leq a} (\mathcal{K}^d_a f)(\omega; t, x)  K_t(\omega; \dd x),
\end{align*}
provided the last term is well-defined.
We note that Taylor's theorem yields that for all \((t, \omega) \in \mathbb{R}_+ \times \Omega\) there exists a constant \(c = c(f, a, t, \omega)\) such that 
\begin{equation}\label{eq: Tay bdd}
\begin{split}
\int_0^t \int_{\|x\| \leq a} \big|( \mathcal{K}^d_a f)(\omega; s, x) \big|  K_s(\omega; \dd x) \dd A_s \leq c \int_0^ t \int_{\|x\| \leq a} \|x\|^2 K_s(\omega; \dd x) \dd A_s < \infty.
\end{split}
\end{equation}
For \(m > 0\) we define
\[
\Theta_m \triangleq \Big\{ (t, \omega) \in [0, m] \times \Omega  \colon \sup_{s \in [0, t]} \|\omega(s-)\| \leq m \Big\}.
\]
\begin{condition}\label{cond: basic}
		\begin{enumerate}
		\item[\textup{(i)}] \emph{Local majoration property of \((b, c, K)\):} For all \(m > 0\) it holds that 
		\begin{align*} 
		\sup_{(t, \omega) \in \Theta_m} \Big(\|b_t (\omega)\| + \|c_t(\omega)\| + \int \big(1\wedge \|x\|^2\big) K_t(\omega; \dd x)\Big) < \infty.
		\end{align*}
		\item[\textup{(ii)}] \emph{Skorokhod continuity property of \((b, c, K)\):} 
		For all \(\alpha \in \Omega\) each of the maps
		\[
		\omega \mapsto b_t(\omega), \widetilde{c}_t(\omega), \int g(x) K_t(\omega;  \dd x),\quad g \in C_1(\mathbb{R}^d), 
		\] 
		is continuous at \(\alpha\) for \(\dd A_t\)-a.a. \(t \in \mathbb{R}_+\).
		\item[\textup{(iii)}] \emph{Local uniform continuity property of \((b, c, K)\):} For all \(t \in \mathbb{R}_+, g \in C_1(\mathbb{R}^d), i, j = 1, \dots, d\) and all Skorokhod compact sets \(K \subset \Omega\) each \(k \in \{\omega \mapsto b^i_t(\omega),  \widetilde{c}^{ij}_t (\omega), \int g(x) K_t(\omega; \dd x)\}\) is uniformly continuous on \(K\), equipped with the local uniform topology, i.e. for all \(\varepsilon > 0\) there exists a \(\delta = \delta(\varepsilon) > 0\) such that for all \(\omega, \alpha \in K\)
		\[
		\sup_{s \in [0, t]} \|\omega(s) - \alpha(s)\| < \delta \quad \Rightarrow \quad |k(\omega) - k(\alpha)| < \varepsilon.
		\]
		\end{enumerate}
\end{condition}
\begin{condition}\label{cond: gl bjp}
	\emph{Big jump property of \(K\):} 
	For all \(m > 0\) we have
	\[
	\lim_{a \nearrow \infty} \sup_{t \in [0, m]} \sup_{\omega \in \Omega} K_t(\omega; \{x \in \mathbb{R}^d \colon \|x\| > a\}) = 0.
	\]
\end{condition}
\begin{condition}\label{cond: Ly1}
	\emph{Lyapunov condition I:} 
	There exists a \(\theta \in \mathbb{R}_+\) such that for all \(a \in (\theta, \infty)\) there exist Borel functions \(V_a \colon \mathbb{R}^d \to (0, \infty), \gamma_a \colon \mathbb{R}_+ \to \mathbb{R}_+\) and \(\beta_a \colon \mathbb{R}_+ \to \mathbb{R}_+\) with the following properties:
	\begin{enumerate} 
		\item[\textup{(a)}] \(V_a \in C^{2}(\mathbb{R}^d)\). 
		\item[\textup{(b)}] \(\int_0^t \gamma_a (s)\dd A_s < \infty\) for all \(t \in \mathbb{R}_+\).
		\item[\textup{(c)}] \(\beta_a\) is increasing and
		\(
		\lim_{n \to \infty} \beta_a(n) = \infty.
		\)
		\item[\textup{(d)}] For all \((t, \omega)\in \mathbb{R}_+ \times \Omega\) we have \(V_a(\omega(t)) \geq \beta_a(\|\omega(t)\|)\)
		and
		\begin{align*}
		\int_0^t \1 \{\gamma_a(s) V_a(\omega(s-)) < (\mathcal{L}_a  V)(\omega; s)\} \dd A_s = 0.
		\end{align*}
	\end{enumerate}
\end{condition}
\begin{condition}\label{cond: LG1}
	 \emph{Linear growth condition I:} There exists a \(\theta \in \mathbb{R}_+\) such that for all \(a \in (\theta, \infty)\) there exists a Borel function \(\gamma_a \colon \mathbb{R}_+ \to \mathbb{R}_+\) 
	such that \(\int_0^t \gamma_a(s) \dd A_s < \infty\) for all \(t \in \mathbb{R}_+\) and for all \(\omega \in \Omega\) and
	for \(\dd A_t\)-a.a. \(t \in \mathbb{R}_+\)
	\begin{align}\label{eq: LG}
	\|b^a_t(\omega)\|^2 + \|\widetilde{c}^a_t(\omega)\| &\leq \gamma_a (t) \Big(1 + \sup_{s \in [0, t]} \|\omega(s -)\|^2\Big).
	\end{align}
\end{condition}
The first main result of this article is the following:
\begin{theorem}\label{theo:1}
	Assume that the Conditions \ref{cond: basic} and \ref{cond: gl bjp} hold and that one of the Conditions \ref{cond: Ly1} and \ref{cond: LG1} holds.
	Then, there exists a probability measure \(P\) on \((\Omega, \mathcal{F})\) such that the coordinate process \(X\) is a \(P\)-semimartingale with local characteristics \((b, c, K; A)\) and initial law \(\eta\).
\end{theorem}
\begin{remark}\label{rem: replace bjc}
In case
\begin{align}\label{eq: unif j bound}
\forall t \in \mathbb{R}_+\  \exists a > 0 \colon \sup_{s \in [0, t]} \sup_{\omega \in \Omega} K_s (\omega; \{x \in \mathbb{R}^d \colon \|x\| > a\}) < \infty,
\end{align}
the big jump condition (Condition \ref{cond: gl bjp}) can be replaced by 
\begin{align}\label{eq: weaker gl bjc}
\lim_{a \nearrow \infty}  \sup_{\omega \in \Omega} K_t (\omega; \{x \in \mathbb{R}^d \colon  \|x\| > a\}) = 0 \text{ for all } t \in \mathbb{R}_+,
\end{align}
see Remark \ref{rem: replace} in Section \ref{sec: sep jumps} below.
\end{remark}
The theorem can be viewed as a generalization of \cite[Theorem IX.2.31]{JS}, which replaces the uniform boundedness assumptions by  local boundedness assumptions and a Lyapunov-type condition or a linear growth condition. Recall that the function \(A\) is assumed to be deterministic and continuous. The continuity of \(A\) is not assumed in \cite[Theorem IX.2.31]{JS}. It implies that any semimartingale with local characteristics \((b, c, K; A)\) is quasi-left continuous, see \cite[Proposition II.2.9]{JS}.
Theorem \ref{theo:1} is proven in Section \ref{sec:p} below.

We need the big jump condition on \(K\) (Condition \ref{cond: gl bjp}) to obtain the existence of our approximation sequence and to show its tightness. In fact, \cite[Theorem VI.4.18]{JS} explains that a (weaker) condition of this type is necessary for tightness of our approximation sequence.
The big jump condition on \(K\) can be replaced by a local big jump condition when the big jumps are also taken into consideration in the Lyapunov and the linear growth condition. 
To state this modification, we introduce some additional notation: For a twice continuously differentiable function \(f \colon \mathbb{R}^d \to  \mathbb{R}\) and \((t, \omega) \in \mathbb{R}_+ \times \Omega\) and \(x \in \mathbb{R}^d\) we set 
\begin{align*}
(\mathcal{K}^df)(\omega; t, x) &\triangleq f(\omega(t-) + x) - f(\omega(t-)) - \sum_{k =  1}^d  \partial_{k} f(\omega(t-)) h^k(x),\\
(\mathcal{K}^l f)(\omega; t) &\triangleq \sum_{k = 1}^d \partial_{k} f(\omega(t-)) b^{k}_t (\omega) + \frac{1}{2} \sum_{k, j = 1}^d \partial^2_{kj} f (\omega(t-)) c^{kj}_t  (\omega),
\end{align*}
and
\begin{align*}
(\mathcal{L} f) (\omega; t) \triangleq  (\mathcal{K}^l f)(\omega; t)+ \int (\mathcal{K}^d f)(\omega; t, x)  K_t(\omega; \dd x),
\end{align*}
provided the last term is well-defined. Furthermore, for \(m > 0\) and \(t \in [0, m]\) we set 
\[
\Theta_m^t \triangleq \big\{ \omega \in \Omega \colon (t, \omega) \in \Theta_m\big\}.
\]
\begin{condition}\label{cond: lbjc}
	 \emph{Local big jump property of \(K\):} For all \(m > 0\) and \(t \in [0, m]\)
	\[
	\lim_{a \nearrow \infty} \sup_{\omega \in \Theta_m^t} K_t (\omega; \{x \in \mathbb{R}^d \colon \|x\| > a\}) = 0.
	\]
\end{condition}
\begin{condition}\label{cond: Ly2}
	\emph{Lyapunov condition II:} 
	There exist Borel functions \(V \colon \mathbb{R}^d \to (0, \infty), \gamma \colon \mathbb{R}_+ \to \mathbb{R}_+\) and \(\beta \colon \mathbb{R}_+ \to \mathbb{R}_+\) with the following properties:
	\begin{enumerate} 
		\item[\textup{(a)}] \(V \in C^{2}(\mathbb{R}^d)\). 
		\item[\textup{(b)}] \(\int_0^t \gamma (s)\dd A_s < \infty\) for all \(t \in \mathbb{R}_+\).
		\item[\textup{(c)}] \(\beta\) is increasing and
		\(
		\lim_{n \to \infty} \beta(n) = \infty.
		\)
		\item[\textup{(d)}] For all \((t, \omega)\in \mathbb{R}_+ \times \Omega\) we have \(V(\omega(t)) \geq \beta(\|\omega(t)\|),\)
		\begin{align}\label{eq: ito int}
		\int_0^t \int \big|(\mathcal{K}^d V)(\omega, s, x) \big| K_s(\omega; \dd x) \dd A_s < \infty,
		\end{align}
		and \begin{align*}
		\int_0^t \1 \{\gamma(s) V(\omega(s-)) < (\mathcal{L}  V)(\omega; s)\} \dd A_s = 0.
		\end{align*}
		\end{enumerate}
\end{condition}

\begin{condition}\label{cond: LG2}
	\emph{Linear growth condition II:} There exits a Borel function \(\gamma \colon \mathbb{R}_+ \to \mathbb{R}_+\) 
	such that \(\int_0^t \gamma(s) \dd A_s < \infty\) for all \(t \in \mathbb{R}_+\) and for all \(\omega \in \Omega\) and
	for \(\dd A_t\)-a.a. \(t \in \mathbb{R}_+\)
	\begin{equation}\label{eq: LG2}
	\begin{split}
	\|b_t(\omega)\|^2 + \|\widetilde{c}_t(\omega)\| + \int \|h'(x)\|^2 K_t(\omega; \dd x) &\leq \gamma (t) \Big(1 + \sup_{s \in [0, t]} \|\omega(s-)\|^2\Big),
	\\
	\int \|h'(x)\| K_t(\omega; \dd x) &\leq \gamma (t) \Big(1 + \sup_{s \in [0, t]} \|\omega(s-)\|^2\Big)^{\frac{1}{2}},
	\end{split}
	\end{equation}
	where \(h'(x) \triangleq x - h(x)\).
\end{condition}

Our second main result is the following:
\begin{theorem}\label{theo:2}
Suppose that the Conditions \ref{cond: basic} and \ref{cond: lbjc} hold and that one of the Conditions \ref{cond: Ly2} and \ref{cond: LG2} holds.
	Then, there exists a probability measure \(P\) on \((\Omega, \mathcal{F})\) such that the coordinate process \(X\) is a \(P\)-semimartingale with local characteristics \((b, c, K; A)\) and initial law \(\eta\).
\end{theorem}
Theorem \ref{theo:2} is also proven in Section \ref{sec:p} below.

\begin{remark}
	In Condition \ref{cond: basic} (i) one can replace \(\Theta_m\) by
	\[
	\Theta_m^* \triangleq \Big\{(t, \omega) \in  [0, m] \times \Omega \colon \sup_{s \in [0, t]} \|\omega(s)\| \leq m\Big\} \subset \Theta_m
	\]
	and in Condition \ref{cond: lbjc} one can replace \(\Theta_m^t\) by 
	\(\{ \omega \in \Omega \colon (t, \omega) \in \Theta_m^* \} \subset \Theta^t_m.
	\)
	Furthermore, in \eqref{eq: LG} and \eqref{eq: LG2} one can replace \(\sup_{s \in [0, t]} \|\omega(s-)\|\) by \(\sup_{s \in [0, t]} \|\omega(s)\|\).
	This follows from part (d) of \cite[Lemma III.2.43]{JS}, which states that for a predictable process \(H\) and all \(t > 0\) and \(\omega, \alpha \in \Omega\) 
	\[\omega (s) = \alpha(s) \text{ for all } s < t  \quad \Rightarrow \quad H_t(\omega) = H_t(\alpha).\]
\end{remark}

Due to this observation, we expect part (i) of Condition \ref{cond: basic} to be close to optimal for a local boundedness condition.
We give some examples for functions having the Skorokhod continuity property and the local uniform continuity property:
\begin{example}\label{ex: cont assp}
	Let \(g \colon \mathbb{R}_+ \times \mathbb{R}^d \to \mathbb{R}\) be a Borel function such that \(x \mapsto g(t, x)\) is continuous for all \(t \in \mathbb{R}_+\). Furthermore, fix \(t > 0\).
	\begin{enumerate}
		\item[(a)] The map \(\omega \mapsto g(t, \omega(t-))\) is continuous at each \(\alpha \in \Omega\) such that \(t \not \in J(\alpha) \triangleq \{s > 0 \colon \alpha(s) \not = \alpha(s-)\}\), see \cite[VI.2.3]{JS}. Recalling that \(A\) is deterministic and continuous and that any \cadlag function has at most countably many discontinuities, we see that the set \(J(\alpha)\) is a \(\dd A_t\)-null set and, consequently, that the Skorokhod continuity property holds. Furthermore, the local uniform continuity property holds. To see this, note that for each compact set \(K \subset \Omega\) there exists a compact set \(K_t \subset \mathbb{R}^d\) such that \(\omega (s) \in K_t\) for all \(\omega \in K\) and \(s \in [0, t]\), see \cite[Problem 16, p. 152]{EK}. Using that continuous functions on compact sets are uniformly continuous, for each \(\varepsilon > 0\) there exists a \(\delta = \delta (\varepsilon) > 0\) such that 
		\[
		x, y \in K_t \colon \|x - y\| < \delta \quad \Rightarrow \quad |g(t, x) - g(t, y)| < \varepsilon.
		\]
		Now, if \(\omega, \alpha \in K\) are such that \(\sup_{s \in [0, t]} \|\omega(s) - \alpha(s)\| < \delta\) we have \(\omega(t-), \alpha (t-) \in K_t\), because \(K_t\) is closed, and \(\|\omega(t-) - \alpha(t-)\| \leq \sup_{s \in [0, t]} \|\omega(s) - \alpha(s)\| < \delta\). Consequently, we have 
		\[
		|g(t, \omega(t-)) - g(t, \alpha(t-))| < \varepsilon. 
		\]
		This shows that the local uniform continuity property holds.
		\item[(b)] If \(g\) is continuous, the map \(\omega \mapsto \int_0^t g(s, \omega(s-))\dd A_s\) is continuous. This follows from the fact that \(\omega \mapsto g(s, \omega(s-))\) is continuous at each \(\alpha \in \Omega\) such that \(s \not \in J(\alpha)\), the dominated convergence theorem and the fact that \(J(\alpha)\) is a \(\dd A_t\)-null set. Furthermore, the map \(\omega \mapsto \int_0^t g(s, \omega(s-))\dd A_s\) has the local uniform continuity property. To see this, let \(K \subset \Omega\) and \(K_t \subset \mathbb{R}^d\) be as in part (a) and fix \(\varepsilon > 0\). Without loss of generality we assume that \(A_t > 0\). Because \(g\) is uniformly continuous on \([0, t] \times K_t\) we find a \(\delta = \delta(\varepsilon) > 0\) such that 
		\[
		x, y \in K_t \colon \|x - y\| < \delta \quad \Rightarrow \quad |g(s, x) - g(s, y)| < \frac{\varepsilon}{2A_t}
		\]
		for all \(s \in [0, t]\).
		Now, for all \(\omega, \alpha \in K\) such that \(\sup_{s \in [0, t]} \|\omega(s) - \alpha(s)\| < \delta\) we have
		\begin{align*}
		\Big| \int_0^t g(s, \omega(s-))\dd A_s &- \int_0^t g(r, \alpha(r-))\dd A_r \Big| \\&\leq \int_0^t \Big| g(s, \omega(s-)) - g(s, \alpha(s-))\Big| \dd A_s < \varepsilon,
		\end{align*}
		which gives the local uniform continuity property.
		\item[(c)] If \(g\) is continuous, the map \(\omega \mapsto \sup_{s \in [0, t]} g(s, \omega(s-))\) is continuous at each \(\alpha \in \Omega\) such that \(t \not \in J(\alpha)\).
		This can be seen with the arguments used in the proof of Lemma \ref{lem: Sk cont} below. Furthermore, the local uniform continuity property holds, which follows with the argument from part (b) and the inequality
		\begin{align*}
		\Big| \sup_{s \in [0, t]} g(s, \omega(s-)) - \sup_{r \in [0, t]} g(r, \alpha(r-))\Big| &\leq \sup_{s \in [0, t]} \Big| g (s, \omega(s-)) - g(s, \alpha(s-))\Big|.
		\end{align*}
	\end{enumerate}
\end{example}
We now comment on the big jump property and the local big jump property.
\begin{example}\label{eq: examp big jump}
	\begin{enumerate}
		\item[\textup{(a)}]
		If \(K_t (\omega; \dd x) = F(\dd x)\) for a L\'evy measure \(F\), then the big jump property of \(K\) (Condition \ref{cond: gl bjp}) holds, because 
		\[
		F(\{x \in \mathbb{R}^d \colon \|x\| > a\}) \to 0 \text{ with } a \nearrow \infty.
		\]
		However, Condition \ref{cond: LG1} can fail, because \(\|h'\|\) might not be \(F\)-integrable, i.e. \(F\) corresponds to a L\'evy process with infinite mean. 
		\item[\textup{(b)}] When we consider a one-dimensional SDE of the type 
		\[
		\dd X_t = g_t (X) \dd L_t, 
		\]
		where \(g\) is predictable and \(L\) is a L\'evy process, then \(\Delta X_t = g_t (X) \Delta L_t\) and, consequently, we consider
		\[
		K_t (G) = \int \1_{G \backslash \{0\}} (g_t (X) y) F(\dd y), \quad G \in \mathcal{B}(\mathbb{R}),
		\]
		where \(F\) is the L\'evy measure corresponding to \(L\).
		In this case, we obtain 
		\[
		K_t (\{x \in \mathbb{R} \colon |x| > a\}) = F( \{y \in \mathbb{R} \colon |y| |g_t(X)| > a\} ).
		\]
		If for \(m \in \mathbb{N}\) there is a constant \(c_m > 0\) such that \(\sup_{(t, \omega) \in [0, m] \times \Omega}|g_t(\omega)| \leq c_m\), then we have 
		\[
		\sup_{t \in [0, m]} \sup_{\omega \in \Omega} F\big( \{y \in \mathbb{R} \colon |y| |g_t(X (\omega))| > a\} \big) \leq F\big(\big\{y \in  \mathbb{R} \colon |y| > \tfrac{a}{c_m}\big\}\big) \to 0 
		\]
		with \(a \nearrow \infty\).
		However, if \(g\) is unbounded, the global big jump property of \(K\) (Condition \ref{cond: gl bjp}) might fail, while the local big jump property of \(K\) (Condition \ref{cond: lbjc}) and the Condition \ref{cond: LG2} might hold. 
		\item[\textup{(c)}] For a jump-diffusion setting we discuss the local big jump property in Section \ref{sec: 2.3} below.
	\end{enumerate}
\end{example}

Next, we provide examples to understand the Lyapunov-type conditions. 
\begin{example}
	\begin{enumerate}
		\item[\textup{(a)}] For  \(V(x) \triangleq 1 + \|x\|^2\) the Lyapunov-type Conditions \ref{cond: Ly1} and \ref{cond: Ly2} correspond to a linear growth condition. For example, if there exists a Borel function \(\gamma \colon \mathbb{R}_+ \to \mathbb{R}_+\) such that for all \(t \in \mathbb{R}_+\) we have \(\int_0^t \gamma(s)\dd A_s < \infty\) and
		\begin{equation}\label{eq: linear growth ly}\begin{split}
		\int_{\|x\| \leq a} \big(\|X_{t-} + x\|^2 - \|&X_{t-}\|^2 - 2\langle X_{t-}, x\rangle \big) K_t (\dd x) \\&+2\langle X_{t-}, b^a_t \rangle  + \|c_t\|\leq \gamma (t) \big(1 + \|X_{t-}\|^2\big), 
		\end{split}
		\end{equation}
		then Condition \ref{cond: Ly1} is satisfied. This linear growth condition is different from Condition \ref{cond: LG1}. On one hand, the growth condition \eqref{eq: linear growth ly} allows an interplay of the coefficients. For example, if \(d = 1\) and 
		\(
		b_t \equiv - X_{t-}^3, c_t \equiv 2X_{t-}^4, K \equiv 0, 
		\) 
		then
		\[
		2  \langle X_{t-}, b_t \rangle + \|c_t\| =  - 2 X_{t-}^4  + 2 X_{t-}^4 = 0 \leq  1 + X_{t-}^2,
		\]
		although \(|b_t|\) and \(|c_t|\) are not of linear growth. 
		On the other hand, Condition \ref{cond: LG1} takes the whole history of the paths into consideration.
		\item[\textup{(b)}] Let us consider the case \(d = 1\) where \(b \equiv K \equiv 0\), i.e. we are looking for a probability measure \(P\) on \((\Omega, \mathcal{F})\) such that the coordinate process is a one-dimensional continuous local \(P\)-martingale with quadratic variation process \(\int_0^\cdot c_s \dd A_s\).  Suppose there exists an \(a > 1\) and a constant \(\zeta < \infty\) such that for all \((t, \omega) \in \mathbb{R}_+ \times \Omega \colon \|\omega(t-)\| < a\) we have
		\(c_t(\omega) \leq \zeta\). 
		Then, the Lyapunov-type Conditions \ref{cond: Ly1} and \ref{cond: Ly2} hold with \(\gamma (t) \triangleq \frac{a^2 \zeta}{\log (a^2)}\) and \(V(x) \triangleq \log(a^2 + |x|^2)\). To see this, note that
		\begin{align*}
		\gamma (t) V(X_t) - (\mathcal{L} V)(t)& = \frac{a^2 \zeta}{\log (a^2)} V(X_t) +\Big(\frac{|X_{t-}|^2 - a^2}{(a^2 + |X_{t-}|^2)^2}\Big) c_t 
		\\&\geq a^2\big(\zeta - c_t \1 \{|X_{t-}| < a\}		\big)
		\geq 0.
		\end{align*}
		In particular, the Conditions \ref{cond: Ly1} and \ref{cond: Ly2} hold when \(c_s (\omega) = \overline{c}(\omega(s-)) \iota_s (\omega)\) for a locally bounded function \(\overline{c}\colon \mathbb{R} \to \mathbb{R}_+\) and a bounded process \(\iota\). 
		This observation can be seen as a generalization of the well-known result that one-dimensional SDEs of the type 
		\begin{align*}
		\dd X_t = \sqrt{\overline{c}(X_t)}\ \dd W_t
		\end{align*}
		have non-exploding weak solutions whenever the coefficient \(\overline{c} \colon \mathbb{R} \to \mathbb{R}_+\) is continuous. 
	\end{enumerate}
\end{example}
\begin{remark} 
	As already indicated in Example \ref{eq: examp big jump}, if we have
	\[
	K_t(\omega; G) = \int \1_{G \backslash  \{0\}} (v(t, \omega, y))F(\dd y),\quad G \in \mathcal{B}(\mathbb{R}^d),
	\] 
	where \(v\) is \(\mathcal{P} \otimes \mathcal{B}(\mathbb{R}^d)\)-measurable and \(F\) is a L\'evy measure on \(\mathbb{R}^d\), then \((b, c, K)\) corresponds to an SDE driven by L\'evy noise, see \cite[Theorem III.2.26]{JS}. Here, \(\mathcal{P}\) denotes the predictable \(\sigma\)-field. In this case, Condition \ref{cond: LG2} is in the spirit of the linear growth conditions from
	\cite[Theorems 14.23, 14.95]{J79} and \cite[Theorem III.2.32]{JS}, which are stated together with local Lipschitz conditions. 
	In particular, Condition \ref{cond: LG2} holds under the following linear growth condition:
	\emph{There exist two Borel functions \(\gamma \colon \mathbb{R}_+ \to \mathbb{R}_+\)
		and \(\theta \colon \mathbb{R}_+ \times \mathbb{R}^d \to \mathbb{R}_+\) such that for all \((t, \omega, y) \in \mathbb{R}_+ \times \Omega \times \mathbb{R}^d\) we have \(\int_0^t( \gamma(s) +\int |\theta(s, x)|^2 F(\dd x))\dd A_s < \infty\) and} 
	\begin{align*}
	\|b_t(\omega)\|^2 + \|\widetilde{c}_t(\omega)\| &\leq \gamma (t) \Big(1 + \sup_{s \in [0, t]} \|\omega(s)\|^2\Big),\\
	\|h'(v(t, \omega, y))\| &\leq \Big(\theta(t, y) \wedge |\theta(t, y)|^2\Big)  \Big(1 + \sup_{s \in [0, t]} \|\omega(s)\|^2\Big)^{\frac{1}{2}}.
	\end{align*}
	Local Lipschitz conditions imply the existence of a local solution. We do not work with a local solution, but construct a solution by approximation. The local Lipschitz conditions also imply uniqueness, which is a property not provided by the approximation argument. 
	Uniform boundedness and continuity conditions for the existence of weak solutions to SDEs driven by semimartingales were proven by Jacod and M\'emin \cite{doi:10.1080/17442508108833169} and Lebedev \cite{doi:10.1080/17442508308833247}. Lebedev \cite{doi:10.1080/17442508408833291} also proved Lyapunov-type conditions.
\end{remark}
As already indicated in the previous remark, Lyapunov-type and linear growth conditions for the existence of weak solutions to SDEs are sometimes combined with conditions implying the existence of a local solution.
Next, we explain the method used by Stroock and Varadhan \cite{SV} to construct a global solution from a local solution and discuss some differences between arguments based on extension and approximation.

The following proposition is a version of Tulcea's extension theorem, which follows from \cite[Theorem 1.1.9]{SV} in the same manner as its continuous analogous \cite[Theorem 1.3.5]{SV} does.
\begin{proposition}\label{prop:extension}
	Let \((\tau_n)_{n \in \mathbb{N}}\) be an increasing sequence of \((\mathcal{F}^o_t)_{t \geq 0}\)-stopping times and let \((P^n)_{n \in \mathbb{N}}\) be a sequence of probability measures on \((\Omega, \mathcal{F})\) such that \(P^n = P^{n +1}\) on \(\mathcal{F}_{\tau_n}^o\) for all \(n \in \mathbb{N}\). If \(\lim_{n \to \infty} P^n(\tau_n \leq t) = 0\) for all \(t \in \mathbb{R}_+\), then there exists a unique probability measure \(P\) on \((\Omega, \mathcal{F})\) such that \(P = P_n\) on \(\mathcal{F}^o_{\tau_n}\) for all \(n \in \mathbb{N}\).
\end{proposition}
Supposing that \((P^n)_{n \in \mathbb{N}}\) is a local solution, the consistency assumption shows that the extension, provided it exists, is a global solution.

Stroock and Varadhan \cite{SV} construct a consistent sequence as in Proposition \ref{prop:extension} under a uniqueness condition. 
In general semimartingale cases, the consistency holds when the sequence \((P^n)_{n \in \mathbb{N}}\) has a local uniqueness property as define in \cite[Definition III.2.37]{JS}.
Local uniqueness is a strong concept of uniqueness, which in particular implies (global) uniqueness. 
In Markovian settings, such as the diffusion setting of Stroock and Varadhan, local uniqueness is implied by the existence of (globally) unique solutions for all degenerated initial laws, see \cite[Theorem III.2.40]{JS}.\footnote{The assumed kernel property in \cite[Theorem III.2.40]{JS} is implied by the uniqueness assumption. This follows from Lemma \ref{lem: P Borel} in Appendix \ref{sec:p2} and Kuratovski's theorem.}
In more general cases, however, local uniqueness is considered to be difficult to show, see the comment in the beginning of \cite[Section III.2d.2]{JS}. In our opinion, using local uniqueness is a natural approach to verify the consistency hypothesis. The approximation argument requires no uniqueness condition. However, it also provides no uniqueness statement.

A version of the convergence criterion \(\lim_{n \to \infty} P^n(\tau_n \leq t) = 0\) from Proposition \ref{prop:extension} is also verified in the tightness argument as presented in Section \ref{sec: tight} below. This is a similarity between the extension and the approximation argument and illustrates that both are soul mates in the point that they prevent a loss of mass.

In some cases it is possible to construct a consistent sequence as in Proposition \ref{prop:extension} without a uniqueness assumption. An example for such a case arises from a local change of measure.
Suppose that \(Q\) is a probability measure and that \(Z\) is a non-negative normalized local \(Q\)-martingale with localizing sequence \((\tau_n)_{n \in \mathbb{N}}\). We define a sequence \((P^n)_{n \in \mathbb{N}}\) by \(P^n = Z_{\tau_n} \cdot Q\), i.e. \(P^n(G) = E^Q [Z_{\tau_n} \1_G]\) for all \(G \in \mathcal{F}\). The consistency follows from the martingale property of \(Z_{\cdot \wedge \tau_n}\) via the optional stopping theorem. Consequently, 
the existence of an extension \(P\) of \((P^n)_{n \in \mathbb{N}}\) follows from Proposition \ref{prop:extension} if 
\[
1 = \lim_{n \to \infty} P^n(\tau_n > t) = \lim_{n \to \infty} E^Q [ Z_{\tau_n} \1 \{\tau_n > t\}] = E^Q [Z_t], \quad t \in \mathbb{R}_+,
\]
which is equivalent to the \(Q\)-martingale property of \(Z\). 
The extension \(P\) is  locally absolutely continuous with respect to \(Q\), because for all \(G \in \mathcal{F}_t^o\) we have \(G \cap \{\tau_n > t\} \in \mathcal{F}^o_{\tau_n}\) and thus
\[
Q(G) = 0 \quad \Rightarrow \quad P(G) = \lim_{n \to \infty} P(G \cap \{\tau_n > t\}) =  \lim_{n \to \infty} P^n(G \cap \{\tau_n > t\})  = 0.
\]
Consequently, if the coordinate process is a \(Q\)-semimartingale, it is also a \(P\)-semimartingale due to \cite[Theorem III.3.13]{JS}.
This argument does not require any form of uniqueness. However, it requires that there exists a probability measure \(Q\) for which the coordinate process is a semimartingale. Furthermore, the structure of the local characteristics under \(P\) is determined by \(Q\) and \(Z\) via Girsanov's theorem (see \cite[Theorem III.3.24]{JS}).
Nevertheless, we think that this method provides a possibility to relax the assumptions in the Theorems \ref{theo:1} and \ref{theo:2}.
Namely, one can apply one of our main results to obtain the probability measure \(Q\) and then deduce the existence of a probability measure \(P\) corresponding to local characteristics which need not to satisfy the continuity conditions formulated in Condition \ref{cond: basic}. We refer to \cite{criensglau2018} for a discussion of the extension method in a general  semimartingale setting.

\subsection{Application: Existence Conditions for Jump-Diffusions}\label{sec: 2.3}
In this subsection we discuss the classical jump-diffusion case as an important example.
Let \(\overline{b} \colon \mathbb{R}_+ \times \mathbb{R}^d \to \mathbb{R}^d\) and \(\overline{c} \colon \mathbb{R}_+ \times \mathbb{R}^d \to \mathbb{S}^{d}\) be Borel functions. Furthermore, let \(\overline{K}_t (x, \dd y)\) be a Borel transition kernel from \(\mathbb{R}_+ \times \mathbb{R}^d\) into \(\mathbb{R}^d\). 
Set for all \(t \in \mathbb{R}_+\)
\begin{align*}
b_t \triangleq \overline{b}(t, X_{t-}),\quad c_t \triangleq \overline{c}(t, X_{t-}),\quad K_t(\dd x) \triangleq \overline{K}_t(X_{t-}, \dd x).
\end{align*}
We assume that for all \(t \in \mathbb{R}_+\) and \(g \in C_1(\mathbb{R}^d)\) the maps \begin{align}\label{eq: cont assp}x \mapsto \overline{b}(t, x), \overline{c}^{ij}(t, x) + \int h^i(y) h^j(y) \overline{K}_t(x, \dd y), \int g(y) \overline{K}_t(x, \dd y)\end{align} are continuous. Then, the Skorokhod continuity property and the local uniform continuity property hold, see part (a) of Example \ref{ex: cont assp}.  Furthermore, we assume that the maps \[(t, x) \mapsto \overline{b}(t, x), \overline{c}(t, x), \int \big(1 \wedge \|y\|^2\big) \overline{K}_t(x, \dd y)\] are locally bounded. Then, the local majoration property holds.
Next, we deduce an existence result from Theorems \ref{theo:1} and \ref{theo:2}. We restrict the statement to the linear growth conditions because these are easier to formulate. Of course, Theorems \ref{theo:1} and \ref{theo:2} also provide Lyapunov-type existence criteria.
\begin{corollary}\label{coro:1}
	In addition to the assumptions above, suppose that one of the following two conditions holds:
	\begin{enumerate}
\item[\textup{(i)}]	For all \(t \in \mathbb{R}_+\) there exists an \(a > 0\) such that 
\[
\sup_{s \in [0, t]} \sup_{x \in \mathbb{R}^d} \overline{K}_s (x, \{y \in \mathbb{R}^d \colon \|y\| > a\}) < \infty,
\]	
for all \(t \in \mathbb{R}_+\)
\begin{align*} 
	\lim_{a \nearrow \infty} \sup_{x \in \mathbb{R}^d} \overline{K}_t (x, \{y \in \mathbb{R}^d \colon \|y\| > a\}) = 0
	\end{align*}
	and there exists a Borel function \(\gamma \colon \mathbb{R}_+ \to \mathbb{R}_+\) such that for all \(t \in \mathbb{R}_+\) and \(x \in \mathbb{R}^d\) we have \(\int_0^t  \gamma(s)\dd  A_s < \infty\) and
	\[
	\|\overline{b}(t, x)\|^2 + \|\overline{c}(t, x)\| + \int \big(1 \wedge \|y\|^2\big) \overline{K}_t (x, \dd y) \leq \gamma(t) \big(1 + \|x\|^2\big).
	\]
	\item[\textup{(ii)}]
	For all \(m > 0\) and \(t \in [0, m]\)
	\begin{align*} 
		\lim_{a \nearrow \infty} \sup_{\|x\| \leq m} \overline{K}_t (x, \{y \in \mathbb{R}^d \colon \|y\| > a\}) = 0
	\end{align*}
and there exits a Borel function \(\gamma \colon \mathbb{R}_+ \to \mathbb{R}_+\) 
such that for all \(t \in \mathbb{R}_+\) and \(x \in \mathbb{R}^d\) we have \(\int_0^t \gamma(s) \dd A_s < \infty\) and
\begin{equation*}
\begin{split}
\|\overline{b} (t, x)\|^2 + \|\widetilde{c}(t, x)\| + \int \|h'(y)\|^2 \overline{K}_t(x, \dd y) &\leq \gamma (t) \Big(1 + \|x\|^2\Big),
\\
\int \|h'(y)\| \overline{K}_t(x, \dd y) &\leq \gamma (t) \Big(1 + \|x\|^2\Big)^{\frac{1}{2}},
\end{split}
\end{equation*}
where \(h'(y) \triangleq y - h(y)\) and 
\[
\widetilde{c}^{ij} (t, x) \triangleq \overline{c}^{ij}(t, x) + \int h^i(y) h^j(y) \overline{K}_t(x, \dd y).
\]
	\end{enumerate}
	Then, there exists a probability measure \(P\) on \((\Omega, \mathcal{F})\) such that the coordinate process \(X\) is a \(P\)-semimartingale with local characteristics \((b, c, K; A)\) and initial law \(\eta\).
\end{corollary}
This existence result can be viewed as a generalization of \cite[Corollary IX.2.33]{JS} where the global boundedness assumptions is replaced by a local boundedness assumptions and a linear growth condition.

As pointed out by a referee, the local big jump condition has a relation to the symbol associated to the candidate triplet. In the following we explain this connection.
We define the symbol associated to \((b, c, K)\) by
\[
q_t(x, \xi) \triangleq - i \langle \overline{b}(t, x), \xi\rangle + \frac{1}{2} \langle \overline{c}(t, x) \xi, \xi\rangle + \int \big(1 - e^{i \langle y, \xi\rangle} + i \langle h(y), \xi\rangle \big)\overline{K}_t(x, \dd y)
\]
for \(t \in \mathbb{R}_+\) and \(x, \xi \in \mathbb{R}^ d\). Recall that \(\overline{b}, \overline{c}\) and \(\overline{K}\) are such that the maps \eqref{eq: cont assp} are continuous.
\begin{proposition}\label{lem: sym lbj}
	For \(t \in \mathbb{R}_+\) consider the following properties:
	\begin{enumerate}
		\item[\textup{(i)}]
For all \(\xi \in \mathbb{R}^d\) the map 
\begin{align}\label{eq: symbol part}
x \mapsto \int \big(1 - e^{i \langle y, \xi\rangle} + i \langle h(y), \xi\rangle \big)\overline{K}_t(x, \dd y)
\end{align}
is continuous.
		\item[\textup{(ii)}]
		For all \(\xi \in \mathbb{R}^d\) the map \(x \mapsto q_t(x, \xi)\) is continuous.
		\item[\textup{(iii)}]
		The local big jump property holds, i.e. 
		\[
		\lim_{a \nearrow \infty} \sup_{\|x\| \leq m} \overline{K}_t (x, \{y \in \mathbb{R}^d \colon \|y\| > a\}) = 0 \text{ for all } m > 0.
		\]
	\end{enumerate}
Then, \textup{(i)} \(\Leftrightarrow\) \textup{(ii)} \(\Rightarrow\) \textup{(iii)}.
\end{proposition}
\begin{proof}
	The equivalence (i) \(\Leftrightarrow\) (ii) is obvious. 
	The implication (ii) \(\Rightarrow\) (iii) follows from \cite[Theorems 4.5.6, 4.5.7]{niels2001pseudo} and \cite[Theorem 4.4]{Schilling1998}.
\end{proof}

This observation has two consequences which we would like to mention.
First, in many applications the continuity of \eqref{eq: symbol part}
 is easy to see and Proposition \ref{lem: sym lbj} can be used to verify the local big jump property. 

Second, in some existence results for time-homogeneous jump-diffusions it is assumed that the symbol is continuous, see, for instance,  \cite[Theorem 3.15]{Hoh98} and \cite[Theorem 2.2]{kuhn18}.
Proposition \ref{lem: sym lbj} relates this assumption to the local big jump condition from Corollary \ref{coro:1}.

\section{Proof of Theorem \ref{theo:1} and Theorem \ref{theo:2}}\label{sec:p}
In view of Proposition \ref{prop: ex} it suffices to show the claim for all degenerated initial laws, i.e. we assume that \(\eta= \delta_z\), where \(z \in \mathbb{R}^d\) is chosen arbitrary. Here \(\delta\) denotes the Dirac measure.
The proof is split into three steps: First, we construct a sequence of probability measures, see Section \ref{sec: Apro}.
Second, we show that the sequence is tight, see Section \ref{sec: tight}. This step requires different arguments under the assumptions of Theorem \ref{theo:1} and Theorem \ref{theo:2}.
Third, we use a martingale problem argument to identify any accumulation point of the sequence as a probability measure under which the coordinate process is a semimartingale with local characteristics \((b, c, K; A)\) and initial law \(\delta_z\), see  Section \ref{sec: MPA}.

In general, we assume that the Conditions \ref{cond: basic} and \ref{cond: lbjc} hold. In case we impose additional assumptions in one of the following sections we indicate these in the beginning.

\subsection{The Approximation Sequence \((P^n)_{n \in \mathbb{N}}\)}\label{sec: Apro}
Let \(\phi^n \colon \mathbb{R} \to [0, 1]\) be a sequence of cutoff functions, i.e. \(\phi^n\in C^\infty_c(\mathbb{R})\) with \(\phi^n(x) = 1\) for \(x \in [-n, n]\) and \(\phi^n(x) = 0\) for \(|x| \geq n + 1\). We define \(X^*_t \triangleq \sup_{s \in [0, t]} \|X_{s-}\|\) 
for \(t \in \mathbb{R}_+\) and note that \(X^*\) is a predictable process, because it is left continuous and adapted.
Set
\begin{align*}
b^n_t &\triangleq  \phi^n(X^*_t) \1 \{t \leq n + 1\} b_t,\\ c^n_t &\triangleq  \phi^n(X^*_{t}) \1 \{t \leq n + 1\}c_t,\\ K^n_t(\dd y) &\triangleq \phi^n(X^*_{t}) \1 \{t \leq n + 1\} K_t(\dd y).
\end{align*}
It is clear that \((b^n, c^n, K^n)\) is a candidate triplet. Fix \(n \in \mathbb{N}\).
Our goal is to apply \cite[Theorem IX.2.31]{JS} to conclude that there exists a probability measure \(P^n\) such that the coordinate process is a \(P^n\)-semimartingale with local characteristics \((b^n, c^n, K^n; A)\) and initial law \(\delta_z\). We proceed by checking the prerequisites of \cite[Theorem IX.2.31]{JS}.

By the local majoration property of the candidate triplet \((b, c, K)\) (Condition \ref{cond: basic} (i)) the modified triplet \((b^n, c^n, K^n)\) has the following global majoration property: 
\begin{align*}
\sup_{t \in \mathbb{R}_+}\sup_{\omega \in\Omega} & \Big( \|b^n_t (\omega)\| + \|c^n_t(\omega)\| + \int \big(1 \wedge \|x\|^2\big) K^n_t(\omega; \dd x)\Big) 
\\&\ \leq \sup_{(t, \omega) \in \Theta_{n + 1}} \Big( \|b_t (\omega)\| + \|c_t(\omega)\| + \int \big(1 \wedge \|x\|^2\big) K_t(\omega; \dd x)\Big) 
< \infty.
\end{align*}

Furthermore, the triplet \((b^n, c^n, K^n)\) has the following modified Skorokhod continuity property: For all \(t \in \mathbb{R}_+\) and \(g \in C_1(\mathbb{R}^d)\) the maps
\[
\omega \mapsto \int_0^t b^n_s (\omega)\dd A_s, \int_0^t \widetilde{c}^n_s (\omega) \dd A_s, \int_0^t \int g(x) K^n_s(\omega; \dd x)\dd A_s
\]
are continuous for the Skorokhod topology. To see this, we first note the following: 
\begin{lemma}\label{lem: Sk cont} The map \(\omega \mapsto \phi^n(X^*_t (\omega))\) is continuous at \(\alpha \in \Omega\) for all \(t \not \in J(\alpha) = \{s > 0 \colon \alpha(s) \not = \alpha(s-)\}\).
\end{lemma}
\begin{proof}
	Let \((\alpha_n)_{n \in \mathbb{N}}\subset \Omega\) such that \(\alpha_n \to \alpha\) as \(n \to \infty\). By \cite[Theorem VI.1.14]{JS} there exists a sequence \((\lambda_n)_{n \in \mathbb{N}}\) of strictly increasing continuous functions \(\mathbb{R}_+ \to \mathbb{R}_+\) such that \(\lambda_n (0) = 0, \lambda_n (t) \nearrow \infty\) as \(t \to \infty\) and for all \(N \in \mathbb{N}\)
	\begin{align}\label{eq:sup skcont}
	\sup_{s \in \mathbb{R}_+} |\lambda_n(s) - s| + \sup_{s \in [0, N]} \|\alpha_n (\lambda_n (s)) - \alpha (s)\| \to 0 
	\end{align}
	as \(n \to \infty\). Now, we have 
	\begin{align*}
	\Big| X^*_t (\alpha_n) - X^*_{\lambda_n^{-1} (t)} (\alpha)\Big| 
	\leq \sup_{s \in [0, \lambda^{-1}_n(t)]} \|\alpha_n (\lambda_n(s)) - \alpha(s)\| \to 0
	\end{align*}
	as \(n \to \infty\) by \eqref{eq:sup skcont}. In case \(t \not \in J(\alpha)\), \eqref{eq:sup skcont} also yields that
	\begin{align*}
	\Big|X^*_{\lambda^{-1}_n (t)}(\alpha) - X^*_t (\alpha)\Big| \to 0\text{ as } n \to \infty.
	\end{align*}
	Thus, \(\omega \mapsto X^*_t(\omega)\) is continuous at \(\alpha\) for all \(t \not \in J(\alpha)\). Because \(\phi^n\) is continuous, this implies the claim.
\end{proof}
Because \cadlag functions have at most countably many discontinuities, for each \(\alpha \in \Omega\) the set \(J(\alpha)\) is at most countable. Thus, because the function \(t \mapsto A_t\) is assumed to be continuous, the set \(J(\alpha)\) is a \(\dd A_t\)-null set. Now, the modified Skorokhod continuity property of \((b^n, c^n, K^n)\) follows from the Skorokhod continuity property of \((b, c, K)\) (Condition \ref{cond: basic} (ii)) and the dominated convergence theorem. 

Finally, we also note that the modified triplet \((b^n, c^n, K^n)\) has the following modified local uniform continuity property: 
\begin{lemma}
	For all \(t \in \mathbb{R}_+, g \in C_1(\mathbb{R}^d), i, j = 1, \dots, d\) and all compact sets \(K \subset \Omega\) any \(k \in \{\omega \mapsto b^{n, i}_t(\omega),  \widetilde{c}^{n,ij}_t (\omega) , \int g(x) K^n_t(\omega; \dd x)\}\) has the uniform continuity property that for all \(\varepsilon > 0\) there exists a \(\delta = \delta(\varepsilon)> 0\) such that for all \(\omega, \alpha \in K\) 
	\[
	\sup_{s \in [0, t]}\|\omega(s) - \alpha(s)\| < \delta\quad \Rightarrow \quad |k(\omega) - k(\alpha)| < \varepsilon.
	\]
\end{lemma}
\begin{proof}
	By the local uniform continuity property of \((b, c, K)\) (Condition \ref{cond: basic} (iii)) it suffices to consider \(k (\omega) = \phi^n(X^*_t(\omega)) g(\omega)\), where \(g\) already has the uniform continuity property and \(|g|\) is bounded by a constant \(\|g\|_\infty > 0\).
	We fix \(\varepsilon > 0\). There exists a \(\delta^* = \delta^*(\varepsilon) > 0\) such that for all \(\omega, \alpha \in K\)
	\[
	\sup_{s \in [0, t]}\|\omega(s) - \alpha(s)\|  <  \delta^*\quad \Rightarrow \quad |g(\omega) - g(\alpha)| < \tfrac{\varepsilon}{2}.
	\]
	Because smooth functions with compact support are Lipschitz continuous, there exists a constant \(L  > 0\) such that 
	\begin{align*}
	|\phi^n(X^*_t(\omega)) - \phi^n(X^*_t(\alpha))| 
	&\leq L \sup_{s \in [0, t]} \|\omega(s) - \alpha (s)\|.
	\end{align*}
	Now, choose \(\delta \triangleq \min (\delta^*, \varepsilon (2 L \|g\|_\infty)^{-1})\). Then, we obtain for all \(\omega, \alpha \in K\colon \sup_{s \in [0, t]} \|\omega(s) - \alpha(s)\| < \delta\) that 
	\begin{align*}
	|k (\omega) - k(\alpha)| &\leq \|g\|_\infty |\phi^n(X^*_t(\omega)) - \phi^n(X^*_t(\alpha))| + |g(\omega) - g(\alpha)| 
	< \tfrac{\varepsilon}{2} + \tfrac{\varepsilon}{2} = \varepsilon.
	\end{align*}
	We conclude that \(k\) has the uniform continuity property. 
\end{proof}

Finally, we note that for all \(t \in \mathbb{R}_+\)
\begin{align*}
\lim_{a \nearrow \infty} \sup_{\omega \in \Omega}\ &K^n_t (\omega; \{x \in \mathbb{R}^d \colon \|x\| > a\})
\\&\leq \lim_{a \nearrow \infty} \sup_{\omega\in \Theta^{t \wedge (n + 1)}_{n + 1}} K_{t \wedge (n + 1)} (\omega; \{x \in \mathbb{R}^d \colon \|x\| > a\}) = 0, 
\end{align*}
by the local big jump property of \(K\) (Condition \ref{cond: lbjc}). In summary, we conclude that the prerequisites of \cite[Theorem IX.2.31]{JS} are fulfilled. Consequently, there exists a probability measure \(P^n\) such that the coordinate process \(X\) is a \(P^n\)-semimartingale with local characteristics \((b^n, c^n, K^n; A)\) and initial law \(\delta_z\). 

\subsection{Tightness of \((P^n)_{n \in \mathbb{N}}\)}\label{sec: tight}
For \(m > 0\) we define the stopping time
\[
\rho_m \triangleq \inf \big(t \in \mathbb{R}_+ \colon \|X_t\| > m\big)\wedge m.
\]
For \(m > 0\) and \(n \in \mathbb{N}\) we define \(P^{n, m}\) to be the law of the stopped process \(X_{\cdot \wedge \rho_m}\) under \(P^n\). 
Our strategy is first to show tightness for \((P^{n, m})_{n \in \mathbb{N}}\) and then to deduce the tightness of \((P^n)_{n \in \mathbb{N}}\) with the help of the Lyapunov and linear growth conditions.

\subsubsection{Tightness of \((P^{n, m})_{n \in \mathbb{N}}\).} \label{sec: tight help seq}
Let \((b^{n, m}, c^{n, m}, K^{n, m}; A)\) be the local characteristics of \(X_{\cdot \wedge \rho_m}\) under \(P_n\). Due to \cite[Lemma 2.3]{KMK}, we have 
\[
b^{n, m} = \1_{\of 0, \rho_m\gs} b^n, \quad c^{n, m} = \1_{\of 0, \rho_m\gs} c^n, \quad K^{n, m} (\dd x) = \1_{\of 0, \rho_m\gs} K^n(\dd x),
\]
where \(\of 0, \rho_m\gs\triangleq \{(t, \omega) \in \mathbb{R}_+ \times \Omega \colon 0 \leq t \leq \rho_m(\omega)\}\).
The tightness of \((P^{n, m})_{n \in \mathbb{N}}\) follows from \cite[Theorem VI.5.10]{JS} once we show the following four conditions:
\begin{enumerate}
	\item[(i)] The sequence \((P^{n, m} \circ X_0^{-1})_{n \in \mathbb{N}}\) is tight.
	\item[(ii)] For all \(t, \epsilon > 0\) we have 
	\[
	\lim_{a \nearrow \infty} \limsup_{n \to \infty} P^{n, m} \Big(\int_0^t K^{n, m}_s (\{x \in \mathbb{R}^d  \colon \|x\| > a\}) \dd A_s > \epsilon \Big) = 0.
	\]
	\item[(iii)] The sequence \((P^{n, m} \circ (\int_0^\cdot b^{n, m}_s \dd A_s)^{-1})_{n \in \mathbb{N}}\) is tight.
	\item[(iv)] For all \(p \in \mathbb{N}\) there exists a deterministic increasing process \(G^{p}\) such that
	\[
	G^{p} - \int_0^\cdot \Big(\sum_{i = 1}^d c^{n, m, ii}_t + \int \Big(\sum_{i = 1}^d |h^i(x)|^2 + (p \|x\| - 1)^+ \wedge 1\Big) K^{n, m}_t(\dd x) \Big)\dd A_t
	\]
	is an increasing process for all \(n \in \mathbb{N}\).
\end{enumerate}

Because \(P^{n, m} \circ X_0^{-1} = \delta_z\) for all \(n, m \in \mathbb{N}\), (i) is trivially satisfied. 
Due to \cite[Fact 2.9, Theorem 2.17]{10.2307/2297861} the map
\[
[0, m] \ni t \mapsto \sup_{\omega \in \Theta_m^t} K_t(\omega; \{x \in \mathbb{R}^d \colon \|x\| > a\})
\]
is universally measurable (see \cite[Definition 2.8]{10.2307/2297861}). Thus, the integral
\[
\int_0^m \sup_{\omega \in \Theta_m^t} K_t(\omega; \{x \in \mathbb{R}^d \colon \|x\| > a\}) \dd A_t
\]
is well-defined. Moreover, we have for all \(a \geq 1\) and \(t \in [0, m]\)
\[
 \sup_{\omega \in \Theta_m^t} K_t(\omega; \{x \in \mathbb{R}^d \colon \|x\| > a\}) \leq \sup_{(s, \omega) \in \Theta_m} \int \big(1 \wedge \|x\|^2\big) K_s(\omega; \dd x) < \infty, 
\]
by local majoration property (Condition \ref{cond: basic} (i)). 
Consequently, we deduce from Chebyshev's inequality, the local big jump property of \(K\) (Condition \ref{cond: gl bjp}) and the dominated convergence theorem that for all \(t ,\varepsilon > 0\)
\begin{align*}
\limsup_{n \to \infty} P^{n, m} &\Big(\int_0^{t} K^{n, m}_s (\{x \in \mathbb{R}^d \colon \|x\| > a\})\dd A_s  > \varepsilon\Big) \\&\leq \frac{1}{\varepsilon} \int_0^m \sup_{\omega \in \Theta^s_m} K_s (\omega; \{x \in \mathbb{R}^d \colon \|x\| > a\}) \dd A_s \to 0 \text{ with } a \nearrow \infty.
\end{align*}
We conclude that (ii) holds.
We set 
\[
\gamma^i \triangleq \sup_{(s, \omega) \in \Theta_m} |b^i_s(\omega)|,\quad i =1, \dots, d.
\]
The local majoration property (Condition \ref{cond: basic} (i)) implies that \(\gamma^i < \infty\) for all \(i = 1, \dots, d\).
Denote by \(\text{Var} (\cdot)\) the variation process. It is easy to see that the process
\[
\sum_{i = 1}^d \gamma^i A - \sum_{i = 1}^d \text{Var}\Big(\int_0^\cdot b^{n, m, i}_s\dd A_s \Big) = \sum_{i = 1}^d \int_0^\cdot \big(\gamma^i - |b^{n, m, i}_s|\big) \dd A_s
\]
is increasing. Thus, we deduce from \cite[Propositions VI.3.35, VI.3.36]{JS} that (iii) holds. 
Similarly, the local majoration property implies that (iv) holds. We conclude from \cite[Theorem VI.5.10]{JS} that \((P^{n, m})_{n \in \mathbb{N}}\) is tight.

\subsubsection{Non-Explosion implies Tightness.}\label{sec: tight and exp} We recall \cite[Theorem 15.47]{HWY}: A sequence \((Q^n)_{n \in \mathbb{N}}\) of probability measures on \((\Omega, \mathcal{F})\) is tight if, and only if, for every \(N \in \mathbb{N}\) and \(\varepsilon, \delta > 0\) there exist \(K, M > 0\) such that 
\begin{align*}
\limsup_{n \to \infty} Q^n \Big(\sup_{t \in [0, N]} \|X_t\| \geq K \Big) &\leq \varepsilon, 
\\
\limsup_{n \to \infty} Q^n \Big(w'(M, X, N) \geq \delta\Big) &\leq \varepsilon, 
\end{align*}
where \(w'\) is the modulus of continuity defined on p. 438 in \cite{HWY}. We only need the following property of \(w'\): For a random time \(\tau\) we have 
\[
w'(M, X, N) = w'(M,X_{\cdot \wedge \tau}, N)
\]
on \(\{N \leq \tau\}\).
Fix \(N \in \mathbb{N}\) and \(\varepsilon, \delta > 0\). Because \((P^{n, m})_{n \in \mathbb{N}}\) is tight, there exist \(K, M > 0\), which depend on \(m\), such that 
\begin{equation}\label{eq: tight1}\begin{split}
\limsup_{n \to \infty} P^n\Big(\sup_{t \in [0, N]} \|X_{t \wedge \rho_m}\| \geq K\Big) &\leq \frac{\varepsilon}{2},\\ 
\limsup_{n \to \infty} P^n\Big( w'(M, X_{\cdot \wedge \rho_m}, N) \geq \delta\Big) &\leq \frac{\varepsilon}{2}.
\end{split}
\end{equation}
Now, we have 
\begin{align*}
P^n \Big(\sup_{t \in [0, N]} \|X_t\| \geq K \Big) 
&\leq P^n \Big(\sup_{t \in [0, N]} \|X_{t \wedge \rho_m}\| \geq K  \Big) + P^n \Big(N > \rho_m \Big),
\\
P^n \Big(w'(M, X, N) \geq \delta\Big) &\leq P^n\Big( w'(M, X_{\cdot \wedge \rho_m}, N) \geq \delta\Big) + P^n\Big(N > \rho_m\Big).
\end{align*}
Thus, using \eqref{eq: tight1}, \((P^n)_{n \in \mathbb{N}}\) is tight if we can chose \(m > 0\) such that 
\[
\limsup_{n \to \infty} P^n\Big(N > \rho_m\Big) \leq \frac{\varepsilon}{2}.
\]
Of course, we would first determine \(m > 0\) and afterwards \(K, M > 0\). 

From this point on the strategies for the conditions from the Theorems \ref{theo:1} and \ref{theo:2} distinguish. To prove Theorem \ref{theo:1} we separate the big jumps, which is a step we do not require in the proof of Theorem \ref{theo:2}.
\subsubsection{Separation of the Big Jumps}\label{sec: sep jumps}
In this section we use ideas from the proof of \cite[Theorem 6.4.1]{liptser1989theory}.
We fix a constant \(a \in (0, \infty]\) which we determine later and \(m > \max(N, 2)\). Set 
\[
Y^a \triangleq \sum_{s \leq \cdot} \Delta X_s \1 \{\|\Delta X_s\| > a\},\quad X^a \triangleq X - Y^a.
\]
Because \(X\) has \cadlag paths, \(\1 \{\|\Delta X_s\| > a\} = 1\) only for finitely many \(s \in [0, t]\). Thus, \(Y^a\) is well-defined.   
Note that for two non-negative random variables \(U\) and \(V\) we have 
\[
P(U + V \geq  2\epsilon)
\leq P(U \geq \epsilon) + P(V \geq \epsilon).
\]
Hence, we obtain
\begin{align*}
P^n \Big(N > \rho_m\Big) 
&\leq P^n\Big(\sup_{s \in [0, N \wedge \rho_m]} \|Y^a_s\|  \geq \frac{m}{2}\Big) + P^n \Big(\sup_{s \in [0, N \wedge \rho_m]} \|X^a_s\|  \geq \frac{m}{2}\Big).
\end{align*}
Clearly, \(\sup_{s \in [0, N \wedge \rho_m]} \|Y^a_s\|\) can only be larger than one in case that at least one jump with norm strictly larger than \(a\) happens before time \(N \wedge \rho_m\), i.e.
\[
\Big\{\sup_{s \in [0, N \wedge \rho_m]} \|Y^a_s\|  \geq 1\Big\} \subseteq \Big\{ \sum_{s \in [0, N \wedge \rho_m]} \1 \{\|\Delta X_s\| > a\}  \geq 1\Big\}.
\]
Thus, we deduce from Lenglart's domination property, see \cite[Lemma I.3.30]{JS}, and Chebyshev's inequality that for all \(\epsilon > 0\)
\begin{align*}
P^n\Big(\sup_{s \in [0, N \wedge \rho_m]} \|Y^a_s\|  \geq 1\Big) &\leq P^n \Big(\sum_{s \in [0, N \wedge \rho_m]} \1 \{\|\Delta X_s\| > a\} \geq 1\Big)
\\&\leq \frac{\epsilon}{7} + P^n \Big(\int_0^{N \wedge \rho_m} K^n_s (\{x \in \mathbb{R}^d \colon \|x\| > a\}) \dd A_s \geq \frac{\epsilon}{7} \Big)
\\&\leq \frac{\epsilon}{7} + \frac{7A_N}{\epsilon} \sup_{s \in [0, N]} \sup_{\omega \in \Omega} K_s (\omega; \{x \in \mathbb{R}^d\colon \|x\| > a\}).
\end{align*}
 In case Condition \ref{cond: gl bjp} is assumed (i.e. in the case of Theorem \ref{theo:1}), we can choose \(a \in ( \theta, \infty)\) independent of \(n\) and \(m\) such that 
\begin{align}\label{eq: choose a}
P^n\Big(\sup_{s \in [0, N \wedge \rho_m]} \|Y^a_s\|  \geq \frac{m}{2}\Big) \leq P^n\Big(\sup_{s \in [0, N \wedge \rho_m]} \|Y^a_s\|  \geq 1\Big) \leq \frac{\varepsilon}{6}.
\end{align}
In case Condition \ref{cond: gl bjp} is not assumed to hold (i.e. in the case of Theorem \ref{theo:2}) we choose \(a \equiv \infty\). Because \(\|Y^\infty\| = 0\), in this case we clearly have 
\[
P^n\Big(\sup_{s \in [0, N \wedge \rho_m]} \|Y^a_s\|  \geq \frac{m}{2}\Big) = P^n\Big(\sup_{s \in [0, N \wedge \rho_m]} \|Y^\infty_s\|  \geq \frac{m}{2}\Big) = 0.
\]
These choices for \(a\) stay fix from now on.
Set 
\[
\zeta_m \triangleq \inf \big(t \in \mathbb{R}_+ \colon \|Y^a_{t \wedge \rho_m}\| > 1\big).
\]
We note that
\begin{align*}
P^n \Big(\sup_{s \in [0, N \wedge \rho_m]} \|X^a_s\| \geq \frac{m}{2}\Big) 
&\leq P^n \Big(\sup_{s \in [0, N \wedge \rho_m \wedge \zeta_m]} \|X^a_s\| \geq \frac{m}{2}\Big) +  P^n \Big(N >\zeta_m\Big)
\\&\leq P^n \Big(\sup_{s \in [0, N \wedge \rho_m \wedge \zeta_m]} \|X^a_s\| \geq \frac{m}{2}\Big) +  \frac{\varepsilon}{6}.
\end{align*}
Consequently, it suffices to choose \(m\) such that 
\[
\limsup_{n \to \infty} P^n \Big(\sup_{s \in [0, N \wedge \rho_m \wedge \zeta_m]} \|X^a_s\| \geq \frac{m}{2}\Big) \leq \frac{\varepsilon}{6}.
\]
\begin{remark}\label{rem: replace}
	We explain the claim from Remark \ref{rem: replace bjc}.
	As in Section \ref{sec: tight help seq}, it follows from \cite[Fact 2.9, Theorem 2.17]{10.2307/2297861} that the integral \[\int_0^N \sup_{\omega \in \Omega} K_s (\omega; \{x \in \mathbb{R}^d\colon \|x\| > a\}) \dd A_s\] is well-defined. In case the global boundedness condition \eqref{eq: unif j bound} holds, the weakend global big jump condition \eqref{eq: weaker gl bjc} and the dominated convergence theorem yield that 
	\[
	\int_0^N \sup_{\omega \in \Omega} K_s (\omega; \{x \in \mathbb{R}^d\colon \|x\| > a\}) \dd A_s \to 0 \text{ with } a \nearrow \infty.
	\]
	Consequently, if \eqref{eq: unif j bound} and \eqref{eq: weaker gl bjc} hold we can choose \(a \in (\theta, \infty)\) such that \eqref{eq: choose a} holds.
\end{remark}
\subsubsection{Non-Explosion under the Lyapunov Conditions}\label{sec: nonexp Ly}
In this section we assume that either the Conditions \ref{cond: gl bjp} and \ref{cond: Ly1} hold or that Condition \ref{cond: Ly2} holds.

In case \(a < \infty\) we deduce from \cite[Theorem II.2.21, Proposition II.2.24]{JS} that the process \(X^a\) is a \(P^n\)-semimartingale with local characteristics \((b^{n, a}, c^n, K^{n, a}; A)\) corresponding to the truncation function \(x \1 \{\|x\| \leq a\}\), where 
\[
b^{n, a}_t \triangleq \phi^n(X^*_t)\1 \{t \leq n + 1\} b^a_t,\quad K^{n, a}_t(\dd x) \triangleq \1 \{\|x\| \leq a\} K^n_t(\dd x), \quad t \in \mathbb{R}_+.
\]
From now on we assume that Condition \ref{cond: Ly2} holds. In case the Conditions \ref{cond: gl bjp} and \ref{cond: Ly1} hold it suffices to replace \(\gamma, V, \beta, \mathcal{L}\) and \(X\) in the following argument by \(\gamma_a, V_a, \beta_a, \mathcal{L}_a\) and \(X^a\).
Set 
\[
Z \triangleq e^{- \int_0^\cdot \gamma(s)\dd A_s} V(X)
\]
and
\[
Y \triangleq Z + \int_0^\cdot e^{- \int_0^s \gamma(u)\dd A_u}\big(\gamma(s) V(X_{s-}) -  (\mathcal{L} V)(s) \phi^n(X^*_{s})\1 \{s \leq n + 1\}\big)\dd A_s.
\]
Because we assume \eqref{eq: ito int} (see \eqref{eq: Tay bdd} for the case where Condition \ref{cond: Ly1} holds), we can deduce from Ito's formula (see, e.g., \cite[Theorem I.4.57]{JS}) and \cite[Lemma I.3.10, Proposition II.1.28]{JS} that \(Y\) is a local \(P^n\)-martingale.
For all \((t, \omega) \in \mathbb{R}_+ \times \Omega\) we have 
\begin{align*}
\int_0^t \1 \{\gamma(s) V(\omega(s-)) <  (\mathcal{L} V)(\omega; s) \phi^n(X^*_{s} (\omega))\1\{s \leq n + 1\}\} \dd A_s =  0, 
\end{align*}
by Condition \ref{cond: Ly2}.
Thus, \(Y \geq Z \geq 0\), which implies that \(Y\) is a non-negative local \(P^n\)-martingale and hence a \(P^n\)-supermartingale by Fatou's lemma.
Because \(\beta\) is increasing with \(\beta(m) \nearrow \infty\) as \(m \to \infty\), we find an \(m > \max(N,2)\) such that 
\[
\beta(k) \geq e^{\int_0^N \gamma(s)\dd A_s} \frac{6V(z)}{\varepsilon}
\]
for all \(k \geq \frac{m}{2}\).
Using that for all \(t \in [0, N]\) 
\[
Y_t \geq Z_t \geq e^{- \int_0^N \gamma(s)\dd A_s} V(X_t) \geq e^{- \int_0^N \gamma(s)\dd A_s} \beta(\|X_t\|), 
\]
we deduce from the supermartingale inequality (see, e.g., \cite[Theorem 1.3.8 (ii)]{KaraShre}) that
\begin{align*}
P^n \Big(\sup_{s \in [0, N]} \|X_s\| \geq \frac{m}{2}\Big) &\leq 
P^n \Big( \sup_{s \in [0, N]} \beta(\|X_s\|) \geq e^{ \int_0^N \gamma(s) \dd A_s} \frac{6 V(z)}{\varepsilon} \Big) 
\\&\leq P^n \Big( \sup_{s \in [0, N]} Y_s \geq \frac{6 V(z)}{\varepsilon} \Big) \leq \frac{\varepsilon V(z)}{6 V(z)} = \frac{\varepsilon}{6}.
\end{align*}
We conclude that \((P^n)_{n \in \mathbb{N}}\) is tight.

\subsubsection{Non-Explosion under Conditions \ref{cond: gl bjp} and \ref{cond: LG1}}\label{sec: LG1}
In this section we assume that the Conditions \ref{cond: gl bjp} and \ref{cond: LG1} hold. 
We use an argument based on Gronwall's lemma. 

Fix \(T > N\) and set 
\[
M^a \triangleq X^a - \int_0^\cdot b^{n, a}_s \dd A_s - X_0.
\]
Due to \cite[Theorem II.2.21, Proposition II.2.24]{JS} the process \(M^a\) is a square-integrable local \(P^n\)-martingale with predictable quadratic variation process
\[
\lle M^a, M^a\rre = \int_0^\cdot \widetilde{c}^{n, a}_s\dd A_s,
\]
where
\[
\widetilde{c}^{n, a}_t \triangleq \phi^n(X^*_t) \1 \{t \leq n + 1\} \widetilde{c}^{a}_t, \quad t \in \mathbb{R}_+.
\]
Thus, using Doob's inequality (see, e.g., \cite[Theorem I.1.43]{JS}), we obtain
\begin{equation}\label{eq: GW1}
\begin{split}
E^{P^n} \Big[ &\sup_{s \in [0, N \wedge \rho_m \wedge \zeta_m]} \|M^a_s\|^2 \Big] \\&\leq 4 E^{P^n} \Big[  \int_0^{N \wedge \rho_m \wedge \zeta_m} \| \widetilde{c}_s^{n, a}\| \dd A_s \Big] 
\\&\leq 4 \int_0^T \gamma_a(s) \dd A_s + 4 \int_0^N \gamma_a(s) E^{P^n} \Big[\sup_{t \in [0, s \wedge \rho_m \wedge \zeta_m]} \|X_{t-}\|^2 \Big] \dd A_s.
\end{split}
\end{equation}
H\"older's inequality yields that 
\begin{equation}\label{eq: GW2}
\begin{split}
\sup_{t \in [0, N \wedge \rho_m \wedge \zeta_m]}&\Big\| \int_0^{t} b_s^{n, a} \dd A_s\Big\|^2 \\&\leq A_T \int_0^{N \wedge \rho_m \wedge \zeta_m} \|b^{n, a}_s\|^2 \dd A_s \\&\leq A_T \int_0^T \gamma_a(s)\dd A_s + A_T \int_0^N \gamma_a(s)\sup_{t \in [0, s \wedge \rho_m \wedge \zeta_m]} \|X_{t-}\|^2 \dd A_s.
\end{split}\end{equation}
By the definition of \(\zeta_m\), we deduce from the inequality \((a_1 + a_2)^2 \leq 2( |a_1|^2 + |a_2|^2)\) that
\begin{align*}
\sup_{t \in [0, s \wedge \rho_m \wedge \zeta_m]} \|X_{t-}\|^2 &\leq 2\ \Big( \sup_{t \in [0, s \wedge \rho_m \wedge \zeta_m]} \|Y^a_{t-}\|^2 +  \sup_{t \in [0, s \wedge \rho_m \wedge \zeta_m]} \|X^a_{t-}\|^2\Big) \\&\leq 2\ \Big(1  + \sup_{t \in [0, s \wedge \rho_m \wedge \zeta_m]} \|X_{t-}^a\|^2\Big).
\end{align*}
Using the inequality \((a_1 + a_2 + a_3)^2 \leq 3 (|a_1|^2 + |a_2|^2 + |a_3|^2)\), we conclude that there exist a constant \(c^* > 0\)  and a \(\dd A_t\)-integrable Borel function \(\iota \colon [0, T] \to \mathbb{R}_+\), which only depend on \(z, T\) and \(\gamma_a\), such that
\begin{align*}
E^{P^n} \Big[ &\sup_{s \in [0, N \wedge \rho_m \wedge \zeta_m]} \|X^a_s\|^2 \Big] \leq c^* + \int_0^N  \iota (s)E^{P^n}\Big[\sup_{t \in [0, s \wedge \rho_m \wedge \zeta_m]} \|X^a_{t-}\|^2 \Big] \dd A_s.
\end{align*}
Applying the Gronwall-type lemma \cite[Theorem 2.4.3]{liptser1989theory} 
we obtain
\[
E^{P^n} \Big[ \sup_{s \in [0, N \wedge \rho_m \wedge \zeta_m]} \|X^a_s\|^2 \Big] \leq c^* e^{\int_0^N \iota(s)\dd A_s}.
\]
Chebyshev's inequality yields that 
\begin{align*}
\limsup_{n \to \infty} P^n \Big(\sup_{s \in [0, N \wedge \rho_m \wedge \zeta_m]} \|X^a_s\| \geq \frac{m}{2}\Big) 
\leq \frac{4 c^* e^{ \int_0^N \iota(s)\dd A_s}}{m^2}.
\end{align*} 
Consequently, we find \(m > \max(N, 2)\) such that 
\[
\limsup_{n \to \infty} P^n \Big(\sup_{s \in [0, N \wedge \rho_m \wedge \zeta_m]} \|X^a_s\| \geq \frac{m}{2}\Big) \leq \frac{\varepsilon}{6}
\]
and therefore we conclude that \((P^n)_{n \in \mathbb{N}}\) is tight.
\subsubsection{Non-Explosion under Condition \ref{cond: LG2}}\label{sec: LG2}
In this section we assume that Condition \ref{cond: LG2} holds. 
The argument is almost identical to the one given in Section \ref{sec: LG1}. The only difference is that we have an additional big jump term. 
Namely, we have
\[
X = X_0 + M + N + \int_0^\cdot b^n_s \dd A_s + \int_0^\cdot \int h'(x) K^n_s (\dd x)\dd A_s,
\]
where 
\begin{align*}
M &\triangleq X - \int_0^\cdot b^n_s\dd A_s - \sum_{s \leq \cdot} h'(\Delta X_s) - X_0, 
\\
N &\triangleq \sum_{s \leq \cdot} h'(\Delta X_s) - \int_0^\cdot \int h'(x) K^n_s (\dd x)\dd A_s.
\end{align*}
Here, \(h\) is the truncation function we fixed from the beginning and \(h'(x) = x - h(x)\). We note that \(\int_0^\cdot \int h'(x) K^n_s (\dd x)\dd A_s\) is well-defined due to Condition \ref{cond: LG2}.
Moreover, \cite[Proposition II.1.28, Theorem II.1.33]{JS} imply that \(N\)
is a square integrable local \(P^n\)-martingale with predictable quadratic variation process
\[
\lle N^i, N^i\rre = \int_0^{\cdot}\int |(h')^i(x)|^2 K^n_s (\dd x) \dd A_s,\quad i = 1, \dots, d.
\]
We deduce from Doob's inequality that 
\begin{align*}
E^{P^n} \Big[ \sup_{s \in [0, N \wedge \rho_m]} \|N_s\|^2 \Big] \leq 4 \int_0^T&\gamma(s) \dd A_s + 4 \int_0^N  \gamma(s) E^{P^n}\Big[\sup_{t \in [0, s \wedge \rho_m]} \|X_{t-}\|^2 \Big] \dd A_s.
\end{align*}
Furthermore, H\"older's inequality yields that
\begin{align*}
\sup_{t \in [0, N \wedge \rho_m]} \Big\| \ &\int_0^{t} \int h'(y) K^n_s(\dd y)\dd A_s \Big\|^2 \\&\leq \Big( \int_0^{N \wedge \rho_m} \int \|h'(y)\| K^n_s(\dd y)\dd A_s \Big)^2
\\&\leq \Big( \int_0^{N \wedge \rho_m}\gamma(s) \Big(1 + \sup_{t \in [0, s\wedge \rho_m]}\|X_{t-}\|^2\Big)^\frac{1}{2} \dd A_s \Big)^2 
\\&\leq \int^T_0 \gamma(s) \dd A_s \int_0^{N} \gamma(s)\Big(1 + \sup_{t \in [0, s\wedge \rho_m]}\|X_{t-}\|^2\Big) \dd A_s
\\&\leq \Big(\int^T_0 \gamma(s) \dd A_s\Big)^2 + \Big(\int_0^T \gamma(s)\dd A_s\Big) \int_0^{N} \gamma(s)\sup_{t \in [0, s\wedge \rho_m]}\|X_{t-}\|^2 \dd A_s.
\end{align*}
Using estimates similar to \eqref{eq: GW1} and \eqref{eq: GW2} and the Gronwall-type lemma \cite[Theorem 2.4.3]{liptser1989theory} yields that 
\[
E^{P^n} \Big[ \sup_{s \in [0, N \wedge \rho_m]} \|X_s\|^2 \Big] \leq c^* e^{\int_0^N \iota(s)\dd A_s}
\]
for a constant \(c^* > 0\) independent of \(n\) and \(m\) and a non-negative Borel function \(\iota\) independent of \(n\) and \(m\) such that \(\int_0^N \iota (s)\dd A_s < \infty\). Chebyshev's inequality completes the proof of the tightness of \((P^n)_{n \in \mathbb{N}}\).
\subsection{Martingale Problem Argument}\label{sec: MPA}
In this section we show that for every accumulation point of \((P^n)_{n \in \mathbb{N}}\) the coordinate process is a semimartingale with local characteristics \((b, c, K; A)\) and initial law \(\delta_z\). 

Let \(P\) be an accumulation point of \((P^n)_{n \in \mathbb{N}}\). Without loss of generality, we assume that \(P^n \to P\) weakly as \(n \to \infty\). Because \(\omega \mapsto \omega(0)\) is continuous, we clearly have \(P \circ X^{-1}_0 = \delta_z\). Set 
\[
\tau_m \triangleq \inf \big(t \in \mathbb{R}_+ \colon \|X_{t-}\| \geq m \text{ or } \|X_t\| \geq m\big),  \quad m > 0, 
\]
and for \(\alpha \in \Omega\) set
\begin{align*}
V(\alpha) &\triangleq \big\{m > 0 \colon \tau_m (\alpha) < \tau_{m + }(\alpha)\big\},\\
V'(\alpha) &\triangleq \big\{m > 0 \colon \Delta \alpha(\tau_m (\alpha)) \not = 0, \|\alpha(\tau_m(\alpha) -)\| = m\big\}.
\end{align*}
Finally, we define
\[
U \triangleq \Big\{m > 0 \colon P \big( \big\{ \omega \in \Omega \colon m \in  V (\omega) \cup V' (\omega)\big\}\big) = 0\Big\}.
\]

Fix \(m \in U\) and denote by \(P_{n, m}\) the law of \(X_{\cdot \wedge \tau_m}\) under \(P^n\) and by \(P_m\) the law of \(X_{\cdot \wedge \tau_m}\) under \(P\).
Due to \cite[Proposition VI.2.12]{JS} and the definition of \(U\), the map \(\omega \mapsto X_{\cdot \wedge \tau_m(\omega)}(\omega)\) is \(P\)-a.s. continuous. Thus, due to the continuous mapping theorem, we have \(P_{n, m} \to P_m\) weakly as \(n \to \infty\). 

Due to \cite[Lemma 2.3]{KMK}, the stopped coordinate process \(X_{\cdot \wedge \tau_m}\) is a \(P^n\)-semimartingale with local characteristics \((\1_{\of 0, \tau_m\gs} b^n, \1_{\of 0, \tau_m\gs} c^n, \1_{\of 0,\tau_m\gs} K^n; A)\). 

Next, we use \cite[Theorem IX.2.11]{JS} to conclude that the stopped coordinate process \(X_{\cdot \wedge \tau_m}\) is a \(P\)-semimartingale with local characteristics \((\1_{\of 0, \tau_m\gs} b, \1_{\of 0, \tau_m\gs} c, \1_{\of 0,\tau_m\gs} K; A)\). For reader's convenience we recall the prerequisites of \cite[Theorem IX.2.11]{JS}:
\begin{enumerate}
\item[(i)]
For all \(t \in \mathbb{R}_+\) and \(g \in C_1(\mathbb{R}^d)\) the maps \[\omega \mapsto \int_0^{t \wedge \tau_m(\omega)} b_s(\omega)\dd A_s, \int_0^{t \wedge \tau_m(\omega)} \widetilde{c}_s(\omega)\dd A_s, \int_0^{t \wedge \tau_m(\omega)} \int g(x) K_s(\omega; \dd x)\dd A_s\] are \(P\)-a.s. continuous.
\item[(ii)]
For all \(t \in \mathbb{R}_+\) and \(g \in C_1(\mathbb{R}^d)\)
\[
\sup_{\omega \in \Omega} \Big(\Big\| \int_0^{t \wedge \tau_m (\omega)} \widetilde{c}_s (\omega) \dd A_s \Big\| + \Big\| \int_0^{t \wedge \tau_m (\omega)} \int g(x) K_s(\omega; \dd  x) \dd A_s \Big\| \Big) < \infty.
\]
\item[(iii)] For all \((k, k^n) \in \{ (b, b^n), (\widetilde{c}, \widetilde{c}^n), (\int g(x) K (\dd x), \int g(x) K^n(\dd x)) \colon g \in C_1(\mathbb{R}^d)\}, t \in \mathbb{R}_+\) and \(\epsilon > 0\) it holds that 
\[
P^n \Big( \Big\| \int_0^{t \wedge \tau_m} (k_s - k^n_s) \dd A_s \Big\| > \epsilon \Big) \to 0 \text{ with } n \to \infty.
\]

\end{enumerate}

Due to the local majoration property (Condition \ref{cond: basic} (i)), the Skorokhod continuity property (Condition \ref{cond: basic} (ii)) and the fact that the map \(\omega \mapsto \tau_m (\omega)\) is \(P\)-a.s. continuous, because \(m \in U\) and \cite[Proposition VI.2.11]{JS}, part (i) holds due to \cite[IX.3.42]{JS}.

Part (ii) follows from the local majoration property (Condition \ref{cond: basic} (i)), because for each \(g \in C_1(\mathbb{R}^d)\) we find a constant \(c^* > 0\) such that \(g(x) \leq c^* (1 \wedge \|x\|^2)\) for all \(x \in \mathbb{R}^d\). 

It remains to explain that (iii) holds. Let \((k, k^n)\) be either \((b, b^n), (\widetilde{c}, \widetilde{c}^n)\) or \((\int g(x) K (\dd x), \int g(x) K^n(\dd x))\), where \(g \in C_1(\mathbb{R}^d)\).
Chebyshev's inequality yields that for all \(t \in \mathbb{R}_+\) and \(\varepsilon > 0\)
\begin{align*}
P^n \Big(\Big\| \int_0^{t \wedge \tau_{m}} (k_s - k^n_s) \dd A_s \Big\| > \varepsilon  \Big) &\leq \frac{1}{\varepsilon} E^{P^n} \Big[\Big\| \int_0^{t \wedge \tau_m} (k_s - k^n_s)\dd A_s\Big\| \Big] 
\\&\leq \frac{1}{\varepsilon} E^{P^n} \Big[ \int_0^{t \wedge \tau_m} \|k_s\| (1 - \phi^n(X^*_{s})) \dd A_s \Big] 
\\&\leq \frac{A_t \sup_{(s, \omega) \in \Theta_{m \vee t}} \|k_s(\omega)\|}{\varepsilon}  \sup_{|x| \leq m} (1 - \phi^n(x))  \to 0 
\end{align*}
with \(n \to \infty\). We conclude that (iii) holds.

In summary, we deduce from \cite[Theorem IX.2.11]{JS} and \cite[Lemma 2.3]{KMK} that the stopped coordinate process \(X_{\cdot \wedge \tau_m}\) is a \(P_m\)-semimartingale with local  characteristics \((\1_{\of 0, \tau_m\gs} b, \1_{\of 0, \tau_m\gs} c, \1_{\of 0,\tau_m\gs} K; A)\). 

Next, we explain that this implies that the stopped coordinate process \(X_{\cdot \wedge \tau_m}\) is also a \(P\)-semimartingale with local  characteristics \((\1_{\of 0, \tau_m\gs} b, \1_{\of 0, \tau_m\gs} c, \1_{\of 0,\tau_m\gs} K; A)\). 
Due to \cite[Theorem II.2.42]{JS} the following are  equivalent:
\begin{enumerate}
	\item[\textup{(i)}] The stopped coordinate process \(X_{\cdot \wedge \tau_m}\) is a \(P\)-semimartingale with local characteristics \((\1_{\of 0, \tau_m\gs} b, \1_{\of 0, \tau_m\gs} c, \1_{\of 0,\tau_m\gs} K; A)\).
	\item[\textup{(ii)}]
	For all bounded \(f \in C^2(\mathbb{R}^d)\) the process 
	\begin{align}\label{eq: MP}
	M^f \triangleq f(X_{\cdot \wedge \tau_m}) - f(X_0) - \int_0^{\cdot \wedge \tau_m} ( \mathcal{L} f)(s)\dd A_s
	\end{align}
	is a local \(P\)-martingale. 
\end{enumerate}
Fix a bounded \(f \in C^2(\mathbb{R}^d)\) and let \(M^f\) be as in \eqref{eq: MP}.
The local majoration property (Condition \ref{cond: basic} (i)) yields that \(M^f\) is bounded on finite time intervals and therefore a martingale whenever it is a local martingale. 
Because \(X_{\cdot \wedge \tau_m}\) is a \(P_m\)-semimartingale with local  characteristics \((\1_{\of 0, \tau_m\gs} b, \1_{\of 0, \tau_m\gs} c, \1_{\of 0,\tau_m\gs} K; A)\), \cite[Theorem II.2.42]{JS} implies that the process \(M^f\) is a \(P_m\)-martingale. Let \(\rho\) be a bounded \((\mathcal{F}^o_t)_{t \geq 0}\)-stopping time.
Due to \cite[Lemma III.2.43]{JS} we have \(M^f_\rho \circ X_{\cdot \wedge \tau_m} = M^f_\rho\). 
Thus, the optional stopping theorem yields that 
\begin{align}\label{eq: st}
E^{P} \big[ M^f_\rho\big] = E^{P_m} \big[M^f_\rho\big] = 0.
\end{align}
Because predictable processes are \((\mathcal{F}_{t-})_{t \geq 0}\)-adapted, see \cite[Proposition I.2.4]{JS}, and \(\mathcal{F}_{t - } \subseteq \mathcal{F}_t^o\) for \(t > 0\), see \cite[p. 159]{JS}, we conclude that \(M^f\) is \((\mathcal{F}^o_t)_{t \geq 0}\)-adapted. Hence, \eqref{eq: st} and \cite[Proposition II.1.4]{RY} yield that \(M^f\) is a \(P\)-martingale for the filtration \((\mathcal{F}^o_t)_{t \geq 0}\). Finally, the backward martingale convergence theorem yields that \(M^f\) is a \(P\)-martingale for the right-continuous filtration \((\mathcal{F}_t)_{t \geq 0}\), too.
We conclude that the stopped coordinate process \(X_{\cdot \wedge \tau_m}\) is a \(P\)-semimartingale with local characteristics \((\1_{\of 0, \tau_m\gs} b, \1_{\of 0, \tau_m\gs} c, \1_{\of 0,\tau_m\gs} K; A)\).

Recall that \(m \in U\) was arbitrary. 
As in the proof of \cite[Proposition IX.1.17]{JS} we see that the complement of \(U\) is at most countable. Consequently, we find  a sequence \((m_k)_{k \in \mathbb{N}}  \subset U\) such that \(m_k \nearrow \infty\) as \(k \to \infty\). In particular, we have \(\tau_{m_k} \nearrow \infty\) as \(k \to \infty\). It follows now from \cite[Theorem II.2.42]{JS} that the coordinate process is a \(P\)-semimartingale with local characteristics \((b, c, K; A)\).
The proof of the Theorems \ref{theo:1} and \ref{theo:2} is complete. \qed

\begin{comment}
The proof of \cite[Theorem IX.2.11]{JS} completely relies on the martingale problem method, i.e. certain processes are identified to be local martingales, which implies the conclusion due to \cite[Theorem II.2.21]{JS}. 
\end{comment}

\section*{Acknowledgements}
The author is grateful to the anonymous referees for many valuable comments which helped to improve the paper.
\appendix
\section{Proof of Proposition \ref{prop: ex}}\label{sec:p2}
We first introduce a martingale problem for semimartingales. 
Let \(\mathscr{C}^+ (\mathbb{R}^d)\) be a countable sequence of test functions as defined in \cite[II.2.20]{JS}. In particular,  any function in \(\mathscr{C}^+ (\mathbb{R}^d)\) is bounded and vanishes around the origin. 
We set \begin{align*}X(h) & \triangleq X - \sum_{s \leq \cdot} (\Delta X_s - h(\Delta  X_s)),\\
M(h) &\triangleq X(h) - \int_0^\cdot b_s \dd A_s - X_0,
\end{align*}
where \(h\) is a truncation function.
Let \(\mathfrak{X}\) be the set of the following processes:
\begin{enumerate}
	\item[(i)] \(M^i(h)\) for \(i = 1, \dots, d\).
	\item[(ii)] \(M^i(h) M^j(h) - \int_0^{\cdot} \widetilde{c}^{ij}_{s} \dd A_s\) for \(i,j = 1, \dots, d\).
	\item[(iii)] \(\sum_{s \leq \cdot} g(\Delta X_s) - \int_0^{\cdot}\int g(x) K_s(\dd  x)\dd A_s\) for \(g \in \mathscr{C}^+(\mathbb{R}^d)\).
\end{enumerate}
For \(n \in \mathbb{N}\) and a \cadlag process \(Y\) we set \[
\tau^Y_n \triangleq \inf \big(t \in \mathbb{R}_+ \colon |Y_{t-}| \geq n \text{ or } |Y_{t}| \geq n\big).
\] 
Moreover, we define
\begin{align*}
\tau^i_n &\triangleq \tau^{Y}_n \text{ with } Y = M^i(h), \\
\tau^{ij}_n &\triangleq \tau^Y_n \text{ with } Y = M^i(h) M^j(h) - \int_0^{\cdot} \widetilde{c}^{ij}_{s} \dd A_s,\\
\tau^{g}_n &\triangleq \tau^Y_n \text{ with } Y = \sum_{s \leq \cdot} g(\Delta X_s) - \int_0^{\cdot}\int g(x) K_s(\dd  x)\dd A_s.
\end{align*}
Let \(\mathfrak{X}_\textup{loc}\)
be the set of the following processes:
\begin{enumerate}
	\item[(i)] \(M^i(h)_{\cdot \wedge \tau^i_n}\) for \(i = 1, \dots, d\) and \(n \in \mathbb{N}\).
	\item[(ii)] \(\big(M^i(h) M^j(h) - \int_0^{\cdot} \widetilde{c}^{ij}_{s} \dd A_s\big)_{\cdot \wedge \tau^i_n \wedge \tau^j_n \wedge \tau^{ij}_n}\) for \(i,j = 1, \dots, d\) and \(n \in \mathbb{N}\).
	\item[(iii)] \(\big(\sum_{s \leq \cdot} g(\Delta X_s) - \int_0^{\cdot}\int g(x) K_s(\dd  x)\dd A_s\big)_{\cdot \wedge \tau^g_n}\) for \(g \in \mathscr{C}^+(\mathbb{R}^d)\) and \(n \in \mathbb{N}\).
\end{enumerate}
We stress that the set \(\mathfrak{X}_\textup{loc}\) is countable.

Due to \cite[Theorem II.2.21]{JS}, \(X\) is a \(P\)-semimartingale with local characterisics \((b, c, K; A)\) and initial law \(\eta\) if, and only if, \(P \circ X^{-1}_0 = \eta\) and all processes in \(\mathfrak{X}\) (or, equivalently, all processes in \(\mathfrak{X}_\textup{loc}\)) are local \(P\)-martingales.

For a bounded function \(f \colon \mathbb{R}^d \to \mathbb{R}^n\) we set \(\|f\|_\infty \triangleq \sup_{x \in \mathbb{R}^d} \|f(x)\|\).
We note that for any \(g \in \mathscr{C}^+(\mathbb{R}^d)\) 
\begin{align*}
|\Delta M^i (h)| + \Big|\Delta \Big( \sum_{s \leq \cdot} g(\Delta X_s) - \int_0^{\cdot}&\int g(x) K_s(\dd  x)\dd A_s\Big)\Big| + \Big| \Delta \Big(\int_0^\cdot \widetilde{c}^{ij}_s\dd A_s\Big)\Big|
\\&\leq 2 \|h^i\|_\infty + 2 \|g\|_\infty +  \|h^i h^j\|_\infty + \|h^i\|_\infty \|h^j\|_\infty, 
\end{align*}
see \cite[II.2.11, Proposition II.2.17]{JS}.
Furthermore, we note that for all \(t \leq \tau^{i}_n \wedge \tau^j_n\)
\begin{align*}
\big| \Delta \big(M^i(h) M^j(h)\big)_{t}\big| &= \big| \Delta M^i(h)_t\Delta M^j (h)_t + M^i(h)_{t-} \Delta M^j(h)_t + M^j(h)_{t-} \Delta M^i(h)_t\big| \\&\leq 
4\|h^i\|_\infty \|h^j\|_\infty + 2n \big(\|h^j\|_\infty + \|h^i\|_\infty\big).
\end{align*}
Hence, because for all \(t \in \mathbb{R}_+\) we have 
\[
\big|Y_{t \wedge \tau^Y_n}\big| \leq n + \big|\Delta Y_{t \wedge \tau^Y_n}\big|, 
\]
we conclude that all processes in \(\mathfrak{X}_\textup{loc}\) are bounded and therefore martingales whenever they are local martingales. Furthermore, because predictable processes are \((\mathcal{F}_{t-})_{t \geq 0}\)-adapted, see \cite[Proposition I.2.4]{JS}, and \(\mathcal{F}_{t - } \subseteq \mathcal{F}_t^o\) for \(t > 0\), see \cite[p. 159]{JS}, all processes in \(\mathfrak{X}\) are \((\mathcal{F}^o_t)_{t \geq 0}\)-adapted.
Because, due to \cite[Proposition 2.1.5]{EK}, the random time \(\tau^Y_n\) is an \((\mathcal{F}^o_t)_{t \geq 0}\)-stopping time whenever \(Y\) is \((\mathcal{F}^o_t)_{t \geq 0}\)-adapted, all processes in \(\mathfrak{X}_\textup{loc}\) are \((\mathcal{F}_t)_{t \geq 0}\)-martingales if, and only if, they are \((\mathcal{F}^o_t)_{t \geq 0}\)-martingales. Here, the implication \(\Rightarrow\) follows from the tower rule and the implication \(\Leftarrow\) follows from the backward martingale convergence theorem. 

In summary, we proved the following:
\begin{lemma}\label{lem: mp}
	For a probability measure \(P\) on \((\Omega, \mathcal{F})\) the coordinate process \(X\) is a \(P\)-semimartingale with local characteristics \((b, c, K; A)\) and initial law \(\eta\) if, and only if, \(P \circ X^{-1}_0 = \eta\) and all processes in \(\mathfrak{X}_\textup{loc}\) are \(P\)-martingales for the filtration \((\mathcal{F}^o_t)_{t \geq 0}\).
\end{lemma}
With this observation at hand we are in the position to prove Proposition \ref{prop: ex}  along the lines of the proof of \cite[Proposition 2]{10.2307/2244838}.

Let \(\mathcal{P}\) be the set of all probability measures \(P\) on \((\Omega,  \mathcal{F})\)  such that the coordinate process is a \(P\)-semimartingale with local characteristics \((b,  c, K; A)\) and initial law \(\delta_z\) for some \(z \in \mathbb{R}^d\).
We consider \(\mathcal{P}\) as a subspace of the Polish space \(\mathscr{P}\) of probability measures on \((\Omega, \mathcal{F})\) equipped with the topology of convergence in distribution. We note that the space \(\mathcal{P}\) is separable and metrizable. 
\begin{lemma}\label{lem: P Borel}
	The set \(\mathcal{P}\) is a Borel subset of \(\mathscr{P}\).
\end{lemma}
\begin{proof}
	Let \(I \triangleq \{P \in \mathscr{P} \colon P \circ X_0^{-1} \in \{\delta_x, x \in \mathbb{R}^d\}\}\) and let \(J\) be the set of all \(P \in \mathscr{P}\) such that
	\begin{align}\label{eq: mb mp}
	E^P \big[ \big(Y_t - Y_s\big)\1_G\big] = 0, 
	\end{align}
	for all \(Y \in \mathfrak{X}_\textup{loc}, 0 \leq s < t < \infty\) and \(G \in \mathcal{F}^o_s\). 
	In \eqref{eq: mb mp} we can restrict ourselves to rational \(0 \leq s < t < \infty\) because of the right-continuity of \(Y\). Furthermore, the \(\sigma\)-field \(\mathcal{F}^o_s = \sigma(X_{r}, r \in [0, s]\cap \mathbb{Q}_+)\) is countable generated, i.e. contains a countable determining class. Thus, in \eqref{eq: mb mp} it also suffices to take only countably many sets from \(\mathcal{F}^o_s\) into consideration. 
	We conclude that \(J\) is Borel due to \cite[Theorem 15.13]{aliprantis2013infinite}. Due to \cite[Theorem 8.3.7]{cohn13} the set \(\{\delta_x, x \in \mathbb{R}^d\}\) is Borel. Thus, since \(P \mapsto P \circ X_0^{-1}\) is continuous, we also conclude that \(I\) is Borel. In view of Lemma \ref{lem: mp}, it follows that \(\mathcal{P} = I \cap J\) is Borel.
\end{proof}
In view of \cite[Theorem A.1.6]{Kallenberg}, the previous lemma implies that \(\mathcal{P}\) is a Borel space in the sense of \cite[p. 456]{Kallenberg}.
Let \(\Phi\colon \mathcal{P} \to \mathbb{R}^d\) be the map such that \(\Phi (P)\) is the starting point associated to \(P \in \mathcal{P}\). 
We claim that \(\Phi\) is continuous and therefore Borel. 
To see this let \((P^n)_{n \in \mathbb{N}}, P \subset \mathcal{P}\) such that \(P^n \to P\) weakly as \(n \to \infty\). Then, we have
\begin{align*}
1 \wedge \|\Phi(P^n) - \Phi(P)\| = E^{P^n} \big[ 1 \wedge \|X_0 -  \Phi(P)\|\big] \to E^P \big[1 \wedge \|X_0 - \Phi(P)\|\big] = 0
\end{align*}
as \(n \to \infty\)
due to the definition of convergence in distribution.
We conclude that \(\Phi\) is continuous.
Furthermore, its graph \(G \triangleq \big\{ (P, \Phi(P)) \colon P \in \mathcal{P}\big\}\) is a Borel subset of \(\mathcal{P} \times \mathbb{R}^d\) due to \cite[Proposition 8.1.8]{cohn13}. 
We have \(\mathcal{B}(\mathcal{P} \times \mathbb{R}^d) = \mathcal{B}(\mathcal{P}) \otimes \mathcal{B}(\mathbb{R}^d)\), see \cite[Proposition 8.1.7]{cohn13}, and 
\(
\bigcup_{P \in \mathcal{P}} \big\{x \in \mathbb{R}^d \colon x = \Psi(P)\big\} = \mathbb{R}^d,
\)
by the assumptions of Proposition \ref{prop: ex}.
Thus, by the section theorem \cite[Theorem A.1.8]{Kallenberg} there exists a Borel map \(x \mapsto P_x\) and a \(\eta\)-null set \(N \in \mathcal{B}(\mathbb{R}^d)\) such that \((P_x, x) \in G\) for all \(x \not \in N\). By the definition of \(G\), for all \(x \not \in N\) the coordinate process is a \(P_x\)-semimartingale with local characteristics \((b, c , K; A)\) and initial law \(\delta_x\).
Clearly, the probability measure \(P_\eta \triangleq \int P_x \eta(\dd x)\) satisfies \(P_\eta \circ X^{-1}_0 = \eta\).
Furthermore, for all \(x \not \in N\) we have 
\[
E^{P_x} \big[ \big(Y_t - Y_s\big) \1_F\big] = 0,
\]
for all \(0 \leq s < t < \infty, F \in \mathcal{F}^o_s\) and \(Y \in \mathfrak{X}_\textup{loc}\). Consequently, integrating and using Lemma \ref{lem: mp} yields that the coordinate process is a \(P_\eta\)-semimartingale with local characteristics \((b, c, K; A)\) and initial law \(\eta\).  This completes the proof. \qed

\bibliographystyle{tfs}

\end{document}